\documentclass[a4paper,11pt]{amsart}
\usepackage{latexsym}
\usepackage{amsmath}
\usepackage{epsfig}
\usepackage{color}
\newtheorem{theorem}{Theorem}
\newtheorem{corollary}[theorem]{Corollary}
\newtheorem{definition}{Definition}
\newtheorem{lemma}[theorem]{Lemma}
\newtheorem{proposition}[theorem]{Proposition}

\newcommand{\NN}{{\rm\bf N}}
\newcommand{\ZZ}{{\rm\bf Z}}
\newcommand{\RR}{{\rm\bf R}}
\newcommand{\EU}{{\rm\bf S}}
\newcommand{\ee}{{\rm e}}
\newcommand{\vv}{{\rm\bf v}}
\newcommand{\ww}{{\rm\bf w}}
\newcommand{\In}{{\text{In}}}
\newcommand{\Out}{{\text{Out}}}
\newcommand{\loc}{{\text{loc}}}

\DeclareMathOperator{\Fix}{Fix}

\newcommand{\wsv}{h_\vv}
\newcommand{\wuw}{h_\ww}
\newcommand{\xmsv}{x_\vv}
\newcommand{\ymsv}{M_\vv}
\newcommand{\xmuw}{x_\ww}
\newcommand{\ymuw}{M_\ww}
\newcommand{\hor}{{\mathcal H}}
\newcommand{\quadr}{{\mathcal S}}

\newcommand{\dpt}{\displaystyle}

\newcommand{\zg}[1]{\langle\gamma_{#1}\rangle}
\def\Qed{\hfill\raisebox{.6ex}{\framebox[2.5mm]{}}\medbreak}

\newcounter{lixo}

\begin{document}

\title[Global bifurcations close to symmetry\qquad \today]{Global bifurcations close to symmetry}

\date{\today }

\author[Isabel S. Labouriau
\quad Alexandre A. P. Rodrigues\quad\today]{
Isabel S. Labouriau
\quad Alexandre A. P. Rodrigues}
\address{Centro de Matem\'atica
da Universidade do Porto
\\
 and
Faculdade de
Ci\^encias, Universidade do Porto \\
Rua do Campo Alegre,
687, 4169-007 Porto, Portugal} 
\email{I.S. Labouriau islabour@fc.up.pt \quad  A.A.P. Rodrigues alexandre.rodrigues@fc.up.pt }
 
\thanks{CMUP is supported by Funda\c{c}\~ao para a Ci\^encia e a Tecnologia --- FCT (Portugal)
with national (MEC) and European (FEDER) funds, under the partnership agreement PT2020.
 A.A.P. Rodrigues was supported by the  grant SFRH/BPD/84709/2012 of FCT. Part of this work has been written during A.A.P. Rodrigues stay in Nizhny Novgorod University supported by the grant RNF 14-41-00044}
 

\begin{abstract}
Heteroclinic  cycles involving two saddle-foci, where the saddle-foci share both invariant manifolds,
occur persistently in some symmetric  differential equations on the 3-dimensional sphere.
We  analyse the dynamics around this type of cycle in
  the case when trajectories near the two equilibria turn in the same direction around a 
1-dimensional connection --- the saddle-foci have the same chirality.
When part of the symmetry is broken, the 2-dimensional invariant manifolds intersect transversely creating a heteroclinic network of Bykov cycles.

  We show that the proximity of symmetry
creates heteroclinic tangencies that coexist with hyperbolic dynamics.  
There are $n$-pulse heteroclinic tangencies --- trajectories  that follow the original cycle $n$ times around before they arrive at the other node. Each $n$-pulse heteroclinic tangency is accumulated by a sequence of $(n+1)$-pulse ones.
This coexists with the suspension of horseshoes defined on an infinite set of disjoint strips, where the first return map is hyperbolic.
We also show how, as the system approaches  full symmetry, the suspended horseshoes are destroyed, creating  regions with infinitely many attracting periodic solutions.
%
%
%
%
\end{abstract}

\maketitle


\section{Introduction}
A Bykov cycle is a heteroclinic 
cycle between two hyperbolic saddle-foci of different Morse index, where one of the connections is transverse and the other is structurally unstable --- see Figure \ref{orientations}.
There are two types of Bykov cycle, depending on the way the flow turns around the two saddle-foci, that determine  the \emph{chirality} of the cycle. Here we study the non-wandering dynamics in the neighbourhood of a Bykov cycle where the two nodes have the same chirality.
This is also studied in \cite{LR}, and the case of different chirality is discussed in \cite{LR3}.
A simplified version of the arguments presented here appears in  \cite{LR_proc}.

\begin{figure}
\begin{center}
\includegraphics[scale=.6]{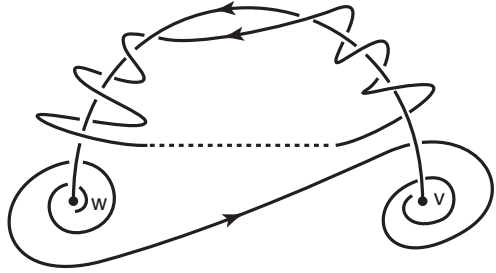}
\end{center}
\caption{\small A Bykov cycle with nodes of the same chirality.
There are two  possibilities for the geometry of the flow around a Bykov cycle  depending on the direction trajectories turn around the connection $[\vv \rightarrow \ww]$. 
We assume here that the nodes have the same chirality:
trajectories turn in the same direction around the connection.
When the endpoints of a nearby trajectory are joined, the closed curve is always linked to the cycle.} 
\label{orientations}
\end{figure}

\subsection{The object of study}
Our starting point is a  fully $(\ZZ_2\times\ZZ_2)$-symmetric 
 differential equation  $\dot{x}=f_0(x)$ in the three-dimensional sphere $\EU^3$
with two saddle-foci that share all the invariant manifolds, of dimensions one and two, both contained in flow-invariant
  submanifolds
 that come from the symmetry.
This forms an attracting heteroclinic network $\Sigma^0$ with a non-empty basin of attraction $V^0$.
 We study the global transition of the dynamics from  this fully symmetric system $\dot{x}=f_0(x)$ to a perturbed system $\dot{x}=f_\lambda(x)$, for a smooth one-parameter family that  breaks  part
  of  the symmetry of the system. 
 For small perturbations the set $V^0$ is still positively invariant.

When $\lambda\neq 0$, the one-dimensional connection persists, due to the remaining symmetry, and the 
two dimensional invariant manifolds intersect transversely, because of the symmetry breaking. 
This gives rise to a network $\Sigma^\lambda$, that consists of a union of Bykov cycles, contained 
in $V^0$.
For partial symmetry-breaking perturbations of $f_0$, we are interested in the dynamics in 
the maximal invariant set contained in $V^0$.
It contains, but does not coincide with, the suspension of horseshoes accumulating on $\Sigma^\lambda$ described in \cite{ACL NONLINEARITY, KLW, LR, Rodrigues2}. 
Here, we show that close to the fully symmetric case it contains infinitely many heteroclinic tangencies.
 Under an additional assumption, we show that $V_0$
 contains attracting limit cycles with long  periods, coexisting with sets with positive entropy.

\subsection{History}
Homoclinic and heteroclinic bifurcations constitute the core of our understanding of complicated recurrent behaviour in dynamical systems.
  This starts with
 Poincar\'e on the late
 $19^{th}$ century, with major subsequent contributions by the schools of Andronov, Shilnikov, Smale and Palis.
 These results rely   on a combination of analytical and geometrical tools used to understand the qualitative behaviour of the dynamics. 

Heteroclinic cycles and networks are flow-invariant sets that can occur robustly in dynamical systems with symmetry, and are frequently associated with intermittent behaviour. 
The rigorous analysis of the  dynamics associated to the structure of the
nonwandering sets close to heteroclinic networks is still a challenge. We refer to \cite{HS} for an overview of heteroclinic bifurcations and for details on the dynamics near different kinds of heteroclinic cycles and networks.

 Bykov cycles have been  found analytically
 in the Lorenz model in \cite{ABS,PY}
  and the nearby dynamics was studied by Bykov in in \cite{Bykov93, Bykov99, Bykov}.
  The point in parameter space where this cycle occurs is called a \emph{T-point}
in \cite{GSpa}.  
Recently, there has been a renewal of interest in this type of heteroclinic bifurcation in the reversible \cite{DIK, DIKS, Lamb2005}, equivariant  \cite{ACL NONLINEARITY,LR, Rodrigues4} and conservative cases \cite{BessaRodrigues}. See also \cite{KLW}. 

The transverse intersection of the two-dimensional invariant manifolds of the two equilibria  implies that the set of  trajectories that remain for all time in a small neighbourhood of the Bykov cycle contains a locally-maximal hyperbolic set admitting a complete description in terms of symbolic dynamics, reminiscent of the results of Shilnikov \cite{Shilnikov67}. 
An obstacle to the global symbolic description  of these trajectories
is the existence of tangencies that lead to the
birth of stable periodic solutions, as described for the homoclinic case in \cite{Afraimovich83, GavS, Newhouse74, Newhouse79, YA}.

All dynamical models with quasi-stochastic attractors were found, either analytically or  by computer simulations, to have tangencies of invariant manifolds  \cite{Afraimovich83, Gonchenko96, Gonchenko2007}. 
 As a rule, the sinks in a quasi-stochastic attractor have very long periods and narrow basins of attraction, and they are hard to observe in applied problems because of the presence of noise \cite{Gonchenko2012}.


Motivated by the analysis of  Lamb \emph{et al} \cite{Lamb2005} in the context of the Michelson system, Bykov cycles have been  considered  by Knobloch {\sl et al} \cite{KLW}.
Using the Lin's method, the authors
 extend the analysis to cycles in spaces of arbitrary dimension, while restricting it
 to trajectories that remain for all time inside a small tubular neighbourhood of the cycle.
 We also refer the reader to \cite{KLW2015}, where the authors consider non-elementary $T$-points in reversible differential equations. The leading eigenvalues at the two equilibria are real and the two-dimensional invariant manifolds meet tangentially. They found chaos in the unfolding of this $T$-point and bifurcations of periodic solutions in the process of annihilation of the shift dynamics.


Bykov cycles  appear in many applications like the Kuramoto-Sivashinsky systems \cite{DIK, Lamb2005}, magnetoconvection \cite{Rodrigues2, Ruck} and travelling waves in reaction-diffusion dynamics \cite{AGH, GH}.

\subsection{Chirality}
  For a Bykov cycle in a 3-dimensional manifold, we say that the nodes have the same chirality if trajectories turn in the same direction around the common 1-dimensional invariant manifold. If they turn in opposite directions we say the nodes have different chirality.
A more formal definition, using links, will be 
 given in Section~\ref{object} below, showing this to be a global topological invariant of the connection that is well defined only in a 3-dimensional ambient space.

These cycles have been studied by different authors who were not aware of the chirality, ignoring what looked like a very small and unimportant choice in local coordinates. 
For instance, Bykov in \cite{Bykov99,Bykov} addresses the case of different chirality without mentioning it explicitly --- see a discussion in Section 7 of \cite{LR3}.
Cycles with the same chirality are treated in \cite{ACL NONLINEARITY,LR} and they occur naturally in reversible differential equations \cite{KLW2015}.
Dynamical features that are irrespective of chirality are described in \cite{KLW}.

Arbitrarily close to Bykov cycles of any chirality there are suspended horseshoes and multi-pulse heteroclinic cycles --- see Theorem~\ref{teorema T-point switching} below.
Around Bykov cycles where the nodes have different chirality, 
heteroclinic tangencies occur generically in trajectories that remain close to the cycle for all time, as shown in  \cite{LR3}.
This is not the case when the nodes have the same chirality, but we show here that heteroclinic tangencies appear when the equations are approaching a more symmetric one.

 \subsection{Symmetry}
 Heteroclinic cycles involving equilibria are   not a generic feature in differential equations,
 but they can be structurally stable in systems which are equivariant under the action of a symmetry group, due to the existence of flow-invariant subspaces \cite{GH}.
 Thus, perturbations that preserve the symmetry will not destroy the cycle.
Explicit examples of equivariant vector fields for which such cycles may be found are reported in \cite{ACL2, ALR, KR, LR3,  MPR,  Rodrigues2, LR2}.  Symmetry, exact or approximate, plays an essential role in the analysis of nonlinear physical phenomena \cite{AGH, GS, Rodrigues2}. 
 It is often incorporated in models either because it occurs approximately in the phenomena being modelled, or because its presence simplifies the analysis.
Since reality has not perfect symmetry, it is desirable to understand the dynamics that are being created under small symmetry breaking perturbations.

 Symmetry plays two roles here. 
First, it creates flow-invariant subspaces where non-transverse heteroclinic connections are persistent, and hence Bykov cycles are robust in this context.
Second,
 we use the proximity of the fully symmetric case to capture more global dynamics.
Symmetry constrains the  geometry of the invariant manifolds of the saddle-foci and allows us some control of their relative positions,
and we find infinitely many heteroclinic tangencies corresponding to trajectories that make an excursion away from the original cycle. 
For Bykov cycles with the same chirality,  tangencies only take place near symmetry.
 
In the analysis of the annihilation of hyperbolic horseshoes associated to tangencies,
on the one hand, the symmetry adds complexity to the problem, because the analysis is not so standard as in \cite{PT, YA}. On the other hand, symmetry simplifies the analytic expression of the return map. 
 It is clear that,  for Bykov cycles of the same chirality,
 the annihilation of hyperbolic horseshoes associated to tangencies only takes place near symmetry. In the fully asymmetric case, the  general study seems to be analitically untreatable.



Many questions remain for future work. One obvious question is to assume different chirality of the nodes. Due to the infinite number of reversals  described in \cite{LR3} , other types
 of bifurcations may occur. The second question is to observe whether the non-wandering set may be reduced to a homoclinic class.

\subsection{This article}

We study the dynamics arising near  a symmetric differential equation with a specific type of heteroclinic network.
We show that when part of the symmetry is broken, the dynamics
 undergoes  a global transition from hyperbolicity 
 coexisting with infinitely many sinks, to the emergence of regular dynamics. 
We discuss the  global bifurcations  that occur as a parameter $\lambda$ is used to break part of the symmetry. 
We complete our results by reducing our problem to a symmetric version of the structure of  Palis and Takens' result  \cite[\S 3]{PT} on homoclinic bifurcations. 
 Being close to  symmetry adds complexity to the dynamics.  Chirality is an essential information in this problem.

This article is organised as follows. 
In Section~\ref{secObject}, after some basic definitions, we describe precisely the object of study and
we review some of our recent  results related to it.
In Section~\ref{secStatement} we state the main results of the present  article. 
 The coordinates and other notation  used in the rest of the article are presented in  Section~\ref{localdyn}, where we also  obtain a geometrical description of the way the flow transforms a curve of initial conditions lying across the stable manifold of an equilibrium. 
In Section \ref{sec tangency}, we prove that there is a sequence of parameter values $\lambda_i$ accumulating on $0$ such that the associated flow has  heteroclinic tangencies.
In Section \ref{bif}, we discuss the geometric constructions that describe the global dynamics near a Bykov cycle as the parameter varies.
 We also describe the limit set that contains nontrivial hyperbolic subsets and we explain how the horseshoes disappear as the system regains full symmetry. 
 We show that under an additional condition this creates infinitely many attracting periodic solutions.

\section{The object of study and preliminary results}\label{secObject}
In the present section, after some preliminary definitions, we
 state the hypotheses for the system under  study together with an overview of  results obtained in \cite{LR}, emphasizing those that will be used to explain the loss of hyperbolicity of the suspended horseshoes and the emergence of heteroclinic tangencies near the cycle.

\subsection{Definitions}
\label{preliminaries} 
Let $f$ be a $C^r$ vector field on the unit three-sphere $\EU^3$, $r\geq 3$, with flow given by the unique solution  $x(t)=\varphi(t,x_{0})\in \EU^{3}$ of 
\begin{equation}
\label{general}
\dot{x}=f_\lambda(x) \qquad  \text{and} \qquad x(0)=x_{0}.
\end{equation}
where $r\geq 3$ and $\lambda \in \RR$. Suppose that $\vv$ and $\ww$ are two hyperbolic saddle-foci of (\ref{general}) with different Morse indices (dimension of the unstable manifold), say 1 and 2.  There is a  {\em heteroclinic cycle }associated to $\{\vv, \ww\}$ if 
$$
W^{u}(\vv)\cap W^{s}(\ww)\neq \emptyset \qquad \text{and} \qquad  W^{u}(\ww)\cap W^{s}(\vv)\neq \emptyset 
$$
where $W^s(\star)$ and $W^u(\star )$ refer to the stable and unstable manifolds of the hyperbolic saddle $\star$, respectively. The terminology $[\vv \to \ww]$ or $[\ww \to \vv]$ denotes a solution contained in $W^{u}(\vv)\cap W^{s}(\ww)$ or $W^{u}(\vv)\cap W^{s}(\ww)$, respectively.
   A \emph{heteroclinic network} is a finite connected union of heteroclinic cycles. 

For $\lambda=0$, there is a 1-dimensional trajectory in $W^{u}(\vv)\cap W^{s}(\ww)$ and a $2$-dimensional connected flow-invariant  manifold contained in $W^{u}(\ww)\cap W^{s}(\vv)$, meaning that there are a continuum of solutions connecting $\ww$ and $\vv$.  For $\lambda\neq 0$, the one-dimensional manifolds of the equilibria coincide and the two-dimensional invariant manifolds have a transverse intersection. This 
 second situation
 is what we call a \emph{Bykov cycle}.




These objects are known to exist in several settings  and are structurally stable within certain classes of 
${\mathbf G}$-equivariant systems, where ${\mathbf G}\subset \textbf{O}(n)$ is a compact Lie group. Here we consider differential equations (\ref{general}) with the equivariance condition:
$$
f_\lambda(\gamma x)=\gamma f_\lambda(x), \qquad \text{for all  } \gamma \in {\mathbf G},  \lambda \in \RR.
$$
 An \emph{ isotropy subgroup} of ${\mathbf G}$ is a set $\widetilde{{\mathbf G}}=\{\gamma\in{\mathbf G}:\ \gamma x=x\}$ for some $x$ in phase space;
we write $\Fix(\widetilde{{\mathbf G}})$ for the vector subspace  of points that are fixed by the elements of 
$\widetilde{{\mathbf G}}$.
For ${\mathbf G}$-equivariant differential equations each subspace $\Fix(\widetilde{{\mathbf G}})$ is flow-invariant.
 The group theoretical methods developed in \cite{GS}  are a powerful tool for the analysis of systems with symmetry. 

Suppose there is a cross-section $S$ to the flow of (\ref{general}), such that $S$ contains a compact invariant set $\Lambda$ where the first return map is well defined and conjugate to a full shift on a countable alphabet,
 we call the  flow-invariant set $\widetilde\Lambda=\{\varphi(t,q)\,:\,t\in\RR,q\in\Lambda\}$ a \emph{suspended horseshoe}.

\subsection{The organising centre}\label{object}
The starting point of the analysis is  a  differential equation $\dot{x}=f_0(x)$ on the unit sphere $\EU^3 =\{X=(x_1,x_2,x_3,x_4) \in \RR^4: ||X||=1\}$
where $f_0: \EU^3 \rightarrow \mathbf{T}\EU^3$ is a
 $C^3$
 vector field  with the following properties:
\begin{enumerate}
\renewcommand{\labelenumi}{(P{\theenumi})}
\item\label{P1} 
The vector field $f_0$  is equivariant under the action of $ {\mathbf G}  =\ZZ_2 \oplus \ZZ_2$ on 
$\EU^3$ induced by the  linear maps on $\RR^4$:
$$
\gamma_1(x_1,x_2,x_3,x_4)=(-x_1,-x_2,x_3,x_4)  
\qquad 
\text{and}
\qquad
\gamma_2(x_1,x_2,x_3,x_4)=(x_1,x_2,-x_3,x_4).$$

\item\label{P2} 
The set $\Fix( \ZZ_2 \oplus \ZZ_2)=\{x \in \EU^3:\gamma_1 x=\gamma_2 x = x \}$ consists of two
equilibria $\vv=(0,0,0,1)$ and $\ww=(0,0,0,-1)$ that  are hyperbolic saddle-foci, 
 where:
\begin{itemize}
\item 
the eigenvalues of $df_0(\vv)$ are
$-C_{\vv } \pm \alpha_{\vv }i$ and $E_{\vv }$ with $\alpha_{\vv } \neq 0$, $C_{\vv }>E_{\vv }>0$;
\item 
the eigenvalues of $df_0(\ww)$ are
$E_{\ww } \pm \alpha_{\ww } i$ and  $-C_{\ww }$ with $\alpha_{\ww } \neq 0$, $C_{\ww }>E_{\ww }>0$.
\end{itemize}
 
\item\label{P3} 
The flow-invariant circle $\Fix(\zg{1})=\{x \in \EU^3:\gamma_1 x = x \}$ consists of the two equilibria 
$\vv$ and $\ww$,  a source and a sink,
respectively, and two heteroclinic trajectories 
 $[\vv \rightarrow \ww]$.

\item\label{P4} 
The  $f_0$-invariant sphere $\Fix(\zg{2})=\{x \in \EU^3:\gamma_2 x = x \}$ consists of the two equilibria $\vv$ and $\ww$, 
and a two-dimensional heteroclinic connection from $\ww$ to $\vv$.
Together with the connections in (P\ref{P3}) this forms a  heteroclinic network that we denote by $\Sigma^0$.
\setcounter{lixo}{\value{enumi}}
\end{enumerate}

Given two small open neighbourhoods $V$ and $W$ of $\vv$ and $\ww$ respectively, consider a piece of trajectory $\varphi$ that starts at $\partial V$, goes into $V$ and then goes once from $V$ to $W$, enters $W$ and ends at $\partial W$. Joining the starting point of $\varphi$ to its end point by a line segment, one obtains a closed curve, the \emph{loop} of $\varphi$.  For almost all starting positions in $\partial V$, the loop of $\varphi$ does not meet the network $\Sigma^0$. If there are arbitrarily small neighbourhoods $V$ and $W$ for which the loop of every trajectory is linked  to $\Sigma_0$, we say that \emph{the nodes have the same chirality} as illustrated in Figure \ref{orientations}. This means that near $\vv$ and $\ww$, all trajectories turn in the same direction around the one-dimensional connections $[\vv\rightarrow \ww]$. This is our last hypothesis on $f_0$:
  \begin{enumerate}
  \renewcommand{\labelenumi}{(P{\theenumi})}
\setcounter{enumi}{\value{lixo}}

\item\label{P6} The saddle-foci $\vv$ and $\ww$ have the same chirality.

\setcounter{lixo}{\value{enumi}}
\end{enumerate}

Condition (P\ref{P6}) means that the curve $\varphi$ and the cycle $\Sigma^0$ cannot be separated by an isotopy.  
This property is persistent under small smooth perturbations of the vector field that preserve the one-dimensional connection.
An explicit example of a family of differential equations where this assumption is valid has been constructed in \cite{LR2}. The rigorous analysis of a case where property (P\ref{P6}) does not hold has been done by the authors in  \cite{LR3}. 

\subsection{The heteroclinic network of the organising centre}

The heteroclinic connections in the network $\Sigma^0$ are contained in fixed point subspaces satisfying the hypothesis (H1) of Krupa and Melbourne \cite{KM1}. Since the inequality $C_\vv C_\ww >E_\vv E_\ww$ holds, the  stability criterion \cite{KM1}  may be applied to $\Sigma^0$ and we have: 

\begin{lemma}\label{propNetworkIstStable}
Under conditions (P\ref{P1})--(P\ref{P4})
the heteroclinic network $\Sigma^0$  is asymptotically stable. 
\end{lemma}

As a consequence of Lemma~\ref{propNetworkIstStable} there exists an open neighbourhood $V^0$ of the network $\Sigma^0$ such that every trajectory starting in $V^0$ remains in it for all positive time and is forward asymptotic to the network. 
The neighbourhood may be taken to have its boundary transverse to the vector field $f_0$. 
The flow associated to any $C^1$-perturbation of $f_0$ that breaks the one-dimensional connection should have some attracting feature. 

\subsection{Breaking  the $\ZZ_2(\zg{1})$-symmetry}

When the symmetry  
 $\ZZ_2(\zg{1})$ is broken, the two one-dimensional heteroclinic connections are destroyed and the cycle $\Sigma^0$ disappears. Each cycle is replaced by a hyperbolic sink that lies close to the original cycle \cite{LR}. For sufficiently small  $C^1$-perturbations, the existence of solutions that go several times around the cycles is ruled out. 

The fixed point hyperplane defined by $\text{Fix}(\zg{2})=\{(x_1, x_2, x_3, x_4) \in \EU^3: x_3=0\}$ divides $\EU^3$ in two flow-invariant connected components, preventing arbitrarily visits to both cycles in $\Sigma^0$. 
Trajectories whose initial condition lies outside the invariant subspaces will approach one of the cycles in positive time. Successive visits to both cycles require breaking this symmetry \cite{LR}.

\subsection{Breaking  the $\ZZ_2(\zg{2})$-symmetry}

From now on, we consider $f_0$ embedded in a generic one-parameter family of vector fields, breaking the $\zg{2}$-equivariance as follows:
  \begin{enumerate}
\renewcommand{\labelenumi}{(P{\theenumi})}
\setcounter{enumi}{\value{lixo}}
\item \label{P5 1/2} The vector fields $f_\lambda: \EU^3 \rightarrow \mathbf{T}\EU^3$ are a $C^3$ family of $\zg{1}$-equivariant $C^3$ vector fields. 
\setcounter{lixo}{\value{enumi}}
\end{enumerate}
Since the equilibria $\vv$ and $\ww$ lie on $\Fix(\zg{1})$ and are hyperbolic, they persist for small $\lambda>0$
and still satisfy Properties (P\ref{P2}) and (P\ref{P3}). 
Their invariant two-dimensional manifolds  generically meet transversely along two trajectories.
The generic bifurcations from a manifold are discussed by Chillingworth \cite{Chillingworth}, under these conditions
 we assume:
  \begin{enumerate}
\renewcommand{\labelenumi}{(P{\theenumi})}
\setcounter{enumi}{\value{lixo}}
\item\label{P5.}  
For $\lambda \neq 0$, the two-dimensional manifolds $W^u(\ww)$ and $W^s(\vv)$ intersect transversely at two trajectories
 that we will call \emph{primary connections}.
Together with the connections in (P\ref{P3}) this forms a Bykov  heteroclinic network that we denote by $\Sigma^\lambda$.
\setcounter{lixo}{\value{enumi}}
\end{enumerate}

The network $\Sigma^\lambda$ consists of four copies of the simplest heteroclinic cycle between two saddle-foci of different Morse indices, where one heteroclinic connection is structurally stable and the other is not, a \emph{Bykov cycle}. 
 Property (P\ref{P5.}) is natural since the heteroclinic connections 
of (P\ref{P5.}), as well as those of  assertion~\eqref{item4} of Theorem~\ref{teorema T-point switching} below, occur at least in symmetric pairs.  
 In what follows, we describe the dynamics near a subnetwork consisting of 
the two primary connections together with one of the unstable connections.
The results obtained concern each one of  the two subnetworks of this type.



For small $\lambda\ne 0$, the neighbourhood $V^0$ is  still  positively invariant and contains the network $\Sigma^\lambda$.
Since the closure of $V^0$ is compact and positively invariant it contains  the $\omega$-limit sets of all its trajectories. 
The union of these limit sets is a maximal invariant set in $V^0$. 
For $f_0$ this is the cycle $\Sigma^0$, by Lemma~\ref{propNetworkIstStable},
whereas for symmetry-breaking perturbations of $f_0$ it contains $\Sigma^\lambda$ but does not coincide with it. 
Our aim is to describe this set  and its sudden appearance.

We proceed to review the dynamics in a small tubular neighbourhood of the cycle.
In order to do this we introduce some concepts.

Let ${\Gamma\subset \Sigma^\lambda}$ be one Bykov cycle involving $\vv$ and $\ww$, and the connections $[\vv\rightarrow\ww]$ and  $[\ww\rightarrow\vv]$ {given by (P\ref{P3}) and (P\ref{P5.})},  respectively. Let $V,W \subset V^0$  be disjoint neighbourhoods of the equilibria as above. 
Consider two  local cross-sections of $\Gamma$ at two points $p$ and $q$ in the connections $[\vv\rightarrow\ww]$ and  $[\ww\rightarrow\vv]$, respectively, with $p, q\not\in V\cup W$.
 Saturating the cross-sections by the flow, one obtains two flow-invariant tubes joining $V$ and $W$ that contain the connections in their interior. 
 We call the union of these tubes with $V$ and $W$ a \emph{tubular neighbourhood} 
 $\mathcal{T}$ of the Bykov cycle. More details will be provided in Section \ref{localdyn}.

\begin{definition}
Let  $V,W \subset V^0$ be two disjoint neighbourhoods of $\vv$ and $\ww$, respectively. A one-dimensional connection
 $[\ww\to\vv]$
 that, after leaving $W$, enters and leaves both $V$ and $W$ precisely $n \in \NN$ times  is called an  \emph{$n$-pulse heteroclinic connection} with respect to $V$ and $W$.  When there is no ambiguity, we omit the expression ``with respect to $V$ and $W$''.
 If  $n>1$
 we call it a \emph{multi-pulse heteroclinic connection}. 
If $W^u(\ww)$ and $W^s(\vv)$ meet tangentially along $[\ww \rightarrow \vv]$, we say that the connection $[\ww \rightarrow \vv]$ is an \emph{$n$-pulse heteroclinic tangency}, otherwise we call it a \emph{transverse $n$-pulse heteroclinic connection}.
\end{definition}
The  primary
 connections $[\ww \rightarrow \vv]$ in $\Sigma^\lambda$ of (P\ref{P5.}) are transverse $0$-pulse heteroclinic connections. With these conventions we have:
 
 \begin{theorem}[\cite{LR}]
\label{teorema T-point switching}
If a vector field $f_0$
satisfies (P\ref{P1})--(P\ref{P6}), then the following properties are satisfied generically by vector fields in an open neighbourhood of $f_0$  in the space of 
$\zg{1}$-equivariant vector fields of class $C^r$ on $\EU^3$, $r\geq 3$:
\begin{enumerate}
\item\label{item0}
there is a heteroclinic network $\Sigma^*$ consisting of four Bykov cycles involving two equilibria 
$\vv$ and $\ww$, two 0-pulse heteroclinic connections $[\vv\rightarrow\ww]$ and two 0-pulse heteroclinic connections
$[\ww\rightarrow\vv]$;
\item \label{item1} 
the only heteroclinic connections from $\vv $ to $\ww $ are those in the Bykov cycles and there are no homoclinic connections;
\item \label{item4} 
any tubular neighbourhood of a Bykov cycle $\Gamma$ in $\Sigma^*$ contains 
infinitely many $n$-pulse heteroclinic connections $[\ww\to\vv]$ for each $n\in\NN$, that accumulate on the cycle;
\item \label{item5} for any tubular neighbourhood $\mathcal{T}$, given a  cross-section $S_q\subset \mathcal{T}$ at a point  $q$ in $[\ww\rightarrow\vv]$, there exist sets of points such that the dynamics of the first return to  $S_q$ is  uniformly hyperbolic and conjugate to a full shift over a finite number of symbols. These sets accumulate on $\Sigma^*$ and   the number of symbols coding the return map tends to infinity as we approach the network.
\end{enumerate}
\end{theorem}

Notice that  assertion~\eqref{item4} of Theorem~\ref{teorema T-point switching}
 implies the existence of a bigger network:
beyond the original transverse connections $[\ww\rightarrow\vv]$, there exist infinitely many subsidiary heteroclinic connections turning around the original Bykov cycle. 
 It also follows from \eqref{item4} and \eqref{item5} that any tubular neighbourhood $\mathcal{T}$ of a Bykov cycle $\Gamma$ in $\Sigma^*$ contains points not lying on $\Gamma$ whose trajectories remain in $\mathcal{T}$ for all time. In contrast to the findings of \cite{Shilnikov65, Shilnikov67, Shilnikov67A}, the suspended horseshoes of (\ref{item5}) arise due to the presence of  two saddle-foci together with transversality of invariant manifolds, and does not depend on any additional condition on the eigenvalues at the equilibria.

A hyperbolic invariant set of a $C^2$-diffeomorphism has zero Lebesgue measure \cite{Bowen75}. Nevertheless, since the authors of \cite{LR} worked in the $C^1$ category, this set of horseshoes might have positive Lebesgue measure. Rodrigues \cite{Rodrigues3} proved that this is not the case:

\begin{theorem}[\cite{Rodrigues3}]
\label{zero measure}
Let $\mathcal{T}$  be a tubular neighbourhood  of one of the Bykov cycles $\Gamma$ of 
Theorem~\ref{teorema T-point switching}.
Then in any cross-section $S_q\subset \mathcal{T}$ the set of initial conditions in $S_q$
that do not leave $\mathcal{T}$ for all time, has zero Lebesgue measure.
\end{theorem}

The shift dynamics does not trap most solutions in the neighbourhood of the cycle. In particular, none of the cycles is Lyapunov stable.

\section{Statement of results}\label{secStatement}
Heteroclinic cycles connecting saddle-foci with a transverse intersection of two-dimensional invariant manifolds imply the existence of hyperbolic suspended horseshoes. In our setting, when $\lambda$ varies close to zero, we expect the creation and the annihilation of these horseshoes. When the symmetry $\ZZ_2(\left\langle \gamma_2\right\rangle)$ is broken, heteroclinic tangencies are reported in the next result.

\begin{theorem}
\label{teorema tangency}
In the set of  families $f_\lambda$ of vector fields satisfying (P\ref{P1})--(P\ref{P5.})
there is a subset ${\mathcal C}$, open  in the $C^3$ topology, for which there is a sequence $\lambda_i>0$ of real numbers, with $\lim_{i\to\infty}\lambda_i=0$ 
such that for $\lambda>\lambda_i$,
there are two 1-pulse heteroclinic connections for the flow of $\dot{x}=f_\lambda(x)$, that collapse into a  1-pulse heteroclinic tangency at $\lambda=\lambda_i$ and then disappear for $\lambda<\lambda_i$. Moreover, the 1-pulse heteroclinic tangency approaches the original $[\ww\to\vv]$ connection when $\lambda_i$ tends to zero.
\end{theorem}

The explicit description of the open set ${\mathcal C}$ is given in Section \ref{sec tangency}, after establishing some notation for the proof. 
Although a tangency may be removed by a small smooth perturbation, the presence of tangencies is densely persistent, a phenomenon similar to the Cocoon bifurcations observed in the Michelson system \cite{DIK, DIKS} observed by Lau.

\begin{theorem}
\label{teorema multitangency}
For a family $f_\lambda$ in the open set ${\mathcal C}$ of Theorem~\ref{teorema tangency}, and for each parameter value $\lambda_i$ corresponding to a 1-pulse heteroclinic tangency, there is a sequence of parameter values $\lambda_{ij}$ accumulating at $\lambda_i$ for which there is a 2-pulse heteroclinic tangency. This property is recursive in the sense that each $n$-pulse heteroclinic tangency is accumulated by $(n+1)$-pulse heteroclinic tangencies for nearby parameter values.
\end{theorem}

Associated to the heteroclinic tangencies of Theorem \ref{teorema multitangency}, suspended horseshoes disappear as $\lambda$ goes to zero.

\begin{theorem}
\label{newhouse}
For a family $f_\lambda$ in the open set ${\mathcal C}$ of Theorem~\ref{teorema tangency},
there is a sequence of closed intervals $\Delta_n= [c_n,d_n] $,
with $0<d_{n+1},c_n<d_n$ and $\dpt\lim_{n\to\infty}d_n=0$, such that as $ \lambda$ decreases in $\Delta_n$, a suspended horseshoe is destroyed.
\end{theorem}

A similar result has been formulated by Newhouse \cite{Newhouse74} and Yorke and Alligood \cite{YA} for the case of two dimensional diffeomorphisms in the context of homoclinic bifurcations
 in dissipative dynamics, with no references to the equivariance.  
A more precise formulation of the result is given in Section~\ref{bif}.
Applying the results of \cite{PT, YA} to this family, we obtain:

\begin{corollary}
\label{Consequences}
For a family $f_\lambda$ in the open set ${\mathcal C}$ of Theorem~\ref{teorema tangency},
and $\lambda$ in one of the intervals $\Delta_n$ of Theorem~\ref{newhouse}
the flow of $\dot{x}=f_\lambda(x)$ undergoes infinitely many saddle-node and period doubling bifurcations.
 \end{corollary}
 
 With the additional hypothesis that the first return map to a suitable cross-section contracts area, we get:

 \begin{corollary}
\label{Sinks}
For a family $f_\lambda$ in the open set ${\mathcal C}$ of Theorem~\ref{teorema tangency},
if the  first return to a transverse section is area-contracting, then
for parameters $\lambda$ in an open subset of $\Delta_n$ with sufficiently large $n$, infinitely many attracting periodic solutions coexist.
 \end{corollary}
 
 In Section~\ref{bif} we also describe a setting where the additional hypothesis holds.
When $\lambda$ decreases, the Cantor set of points of the horseshoes that remain near the cycle is losing topological entropy, as the set loses hyperbolicity, a phenomenon similar to that described in \cite{Gonchenko2007}. 
For $\lambda \approx 0$, return maps to appropriate domains close to the tangency are conjugate to H\'enon-like maps \cite{Colli, Kiriki, MV}. 
As $\lambda \rightarrow 0$, in $V^0$, infinitely many wild attractors may coexist with suspended horseshoes that are being destroyed.

The complete description of the dynamics near these bifurcations is an unsolvable problem:  arbitrarily small perturbations of any
differential equation with a quadratic heteroclinic tangency may lead
to the creation of new tangencies of higher order, and to the birth of quasi-stochastic attractors \cite{GST1993, GTS2001, GST2007, Gonchenko2012}.

\section{Local geometry and transition maps}
  \label{localdyn}
We analyse the dynamics near the network by deriving local maps that approximate the dynamics
near and between the two nodes in the network. 
In this section we establish the notation that will be used in the rest of the article and the expressions for the local maps. We start with appropriate coordinates near the two saddle-foci.

\subsection{Local coordinates}
In order to describe the dynamics around the Bykov cycles of $\Sigma^\lambda$, we use the local coordinates near the equilibria $\vv$ and $\ww$ introduced by Ovsyannikov and Shilnikov \cite{OS}.
Without loss of generality we assume that $\alpha_\vv=\alpha_\ww=1$. 

 \begin{figure}
\begin{center}
\includegraphics[height=6cm]{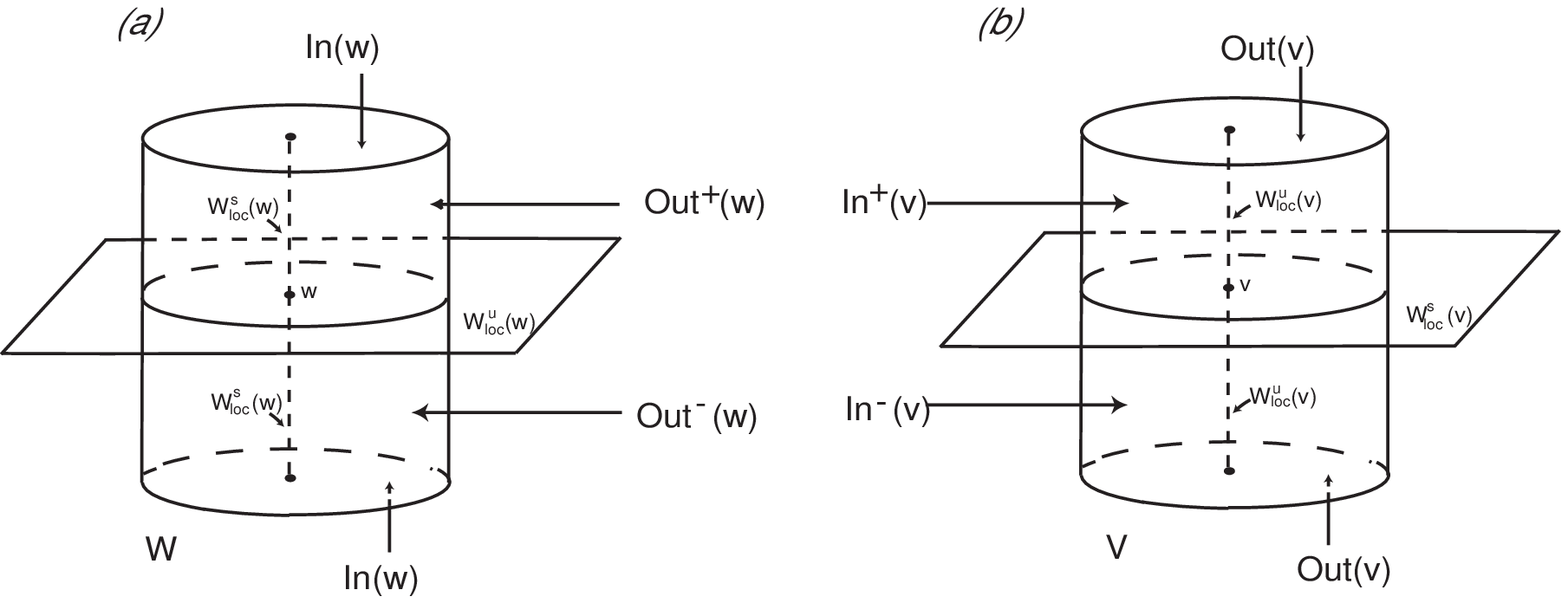}
\end{center}
\caption{\small  Cylindrical neighbourhoods of the saddle-foci $\ww$ (a) and $\vv$ (b). }
\label{neigh_vw}
\end{figure}

In these coordinates, we  consider cylindrical neighbourhoods  $V$ and $W$  in ${\EU}^3$ of $\vv $ and $\ww $, respectively, of radius $\rho=\varepsilon>0$ and height $z=2\varepsilon$ --- see Figure \ref{neigh_vw}.
After a linear rescaling of the variables, we  may also assume that  $\varepsilon=1$.
Their boundaries consist of three components: the cylinder wall parametrised by $x\in \RR\pmod{2\pi}$ and $|y|\leq 1$ with the usual cover $$ (x,y)\mapsto (1 ,x,y)=(\rho ,\theta ,z)$$ and two discs, the top and bottom of the cylinder. We take polar coverings of these disks $$(r,\varphi )\mapsto (r,\varphi , \pm 1)=(\rho ,\theta ,z)$$
where $0\leq r\leq 1$ and $\varphi \in \RR\pmod{2\pi}$.
The local stable manifold of $\vv$, $W^s(\vv)$, corresponds to the circle parametrised by $ y=0$. In $V$ we use the following terminology suggested in Figure \ref{neigh_vw}:
\begin{itemize}
\item
$\In(\vv)$, the cylinder wall of $V$,  consisting of points that go inside $V$ in positive time;
\item
$\Out(\vv)$, the top and bottom of $V$,  consisting of points that go outside $V$ in positive time.
\end{itemize}
We denote by $\In^+(\vv)$ the upper part of the cylinder, parametrised by $(x,y)$, $y\in[0,1]$ and by $\In^-(\vv)$ its lower part.

The cross-sections obtained for the linearisation around $\ww$ are dual to these. The set $W^s(\ww)$ is the $z$-axis intersecting the top and bottom of the cylinder $W$ at the origin of its coordinates. The set 
$W^u(\ww)$ is parametrised by $z=0$, and we use:

\begin{itemize}
\item
$\In(\ww)$, the top and bottom of $W$,  consisting of points that go inside $W$ in positive time;
\item
$\Out(\ww)$,  the cylinder wall  of $W$,  consisting of points that go inside $W$ in negative time, with $\Out^+(\ww)$ denoting its upper part, parametrised by $(x,y)$, $y\in[0,1]$ and $\Out^-(\ww)$  its lower part.
\end{itemize}

We will denote by $W^u_{\loc}(\ww)$ the portion of $W^u(\ww)$  that goes from $\ww$ up to $\In(\vv)$ not intersecting the interior of $V$ and by $W^s_{\loc}(\vv)$  the portion of $W^s(\vv)$ outside $W$ that goes directly  from $\Out(\ww)$ into $\vv$. The flow is transverse to these cross-sections and the boundaries of $V$ and of $W$ may be written as the closure of  $\In(\vv) \cup \Out (\vv)$ and  $\In(\ww) \cup \Out (\ww)$, respectively.

\subsection{Local maps near the saddle-foci}
Following \cite{Deng, OS}, the trajectory of  a point $(x,y)$ with $y>0$ in $\In(\vv)$, leaves $V$ at
 $\Out(\vv)$ at
\begin{equation}
\Phi_{\vv }(x,y)=\left(y^{\delta_\vv} + S_1(x,y; \lambda),-\frac{\ln y}{E_\vv}+x+S_2(x,y; \lambda) \right)=(r,\phi)
\qquad \mbox{where}\qquad 
\delta_\vv=\frac{C_{\vv }}{E_{\vv}} > 1,
\label{local_v}
\end{equation}
where $S_1$ and $S_2$ are smooth functions which depend on $\lambda$ and satisfy:
\begin{equation}
\label{diff_res}
\left| \frac{\partial^{k+l+m}}{\partial x^k \partial x^l  \partial \lambda ^m } S_i(x, y;\lambda)
\right| \leq C y^{\delta_\vv + \sigma - l},
\end{equation}
and $C$ and $\sigma$ are positive constants and $k, l, m$ are non-negative integers. Similarly, a point $(r,\phi)$ in $\In(\ww) \backslash W^s_{\loc}(\ww)$, leaves $W$ at $\Out(\ww)$ at
\begin{equation}
\Phi_{\ww }(r,\varphi )=\left(-\frac{\ln r}{E_\ww}+\varphi+ R_1(r,\varphi ; \lambda),r^{\delta_\ww }+R_2(r,\varphi; \lambda )\right)=(x,y)
\qquad \mbox{where}\qquad 
\delta_\ww=\frac{C_{\ww }}{E_{\ww}} >1 
\label{local_w}
\end{equation}
where $R_1$ and $R_2$ satisfy a  condition similar  to (\ref{local_v}). The terms $S_1$, $S_2$,  $R_1$, $R_2$ correspond to asymptotically small terms that vanish when $y$ and $r$ go to zero. A better estimate under a stronger eigenvalue condition has been obtained in \cite[Prop. 2.4]{Homburg}.

\subsection{Transition map along the connection $[\vv\rightarrow\ww]$}

Points in $\Out(\vv)$ near $W^u_{\loc}(\vv)$ are mapped into $\In(\ww)$  in a  flow-box along the each one of the connections  $[\vv\rightarrow\ww]$. 
Without loss of generality, we will assume that the transition $\Psi_{\vv \rightarrow \ww}: \Out(\vv) \rightarrow \In(\ww)$ does not depend on $\lambda$ and is modelled by the identity, which is compatible with hypothesis (P\ref{P6}).
Using a more general form for $\Psi_{\vv \rightarrow \ww}$ would complicate the calculations without any change in the final results.

The coordinates on $V$ and $W$ are chosen to have $[\vv\rightarrow\ww]$
connecting points with $z>0$ in  $V$ to points with $z>0$ in $W$.
We will denote by $\eta$ the map 
$\eta=\Phi_{\ww } \circ \Psi_{\vv \rightarrow \ww} \circ \Phi_{\vv }$.
Up to higher order terms, from \eqref{local_v} and \eqref{local_w}, its expression in local coordinates, for $y>0$, is
\begin{equation}
\label{eqeta}
\eta(x,y)=\left(x-K \ln y , y^{\delta} \right)
\qquad\mbox{with}\qquad
\delta=\delta_\vv \delta_\ww>1 
\qquad\mbox{and}
\qquad K= \frac{C_\vv+E_\ww}{E_\vv E_\ww} > 0.
\end{equation}

The choice of $\Psi_{\vv \rightarrow \ww}$  as the identity reflects the fact that the node have the same chirality.
In the case where the nodes have different chirality, them map $\Psi_{\vv \rightarrow \ww}$ reverses orientation.
This affects the form of the map $\eta$ and any subsequent results that use it.

 \subsection{Geometry of the transition map}
 Consider a cylinder $C$ parametrised by a covering $(\theta,y )\in  \RR\times[-1,1]$,  where $\theta $ is periodic.
A \emph{helix} on the cylinder $C$ \emph{accumulating on the circle} $y=0$ is a curve on $C$ without self-intersections, that is the image,  by the parametrisation $(\theta,y)$,
of a continuous map
$H:(a,b)\rightarrow \RR\times[-1,1]$, $H(s)=\left(H_\theta(s),H_y(s)\right)$, satisfying:
$$
\lim_{s\to a^+}H_\theta(s)=\lim_{s\to b^-}H_\theta(s)=\pm\infty
\qquad
\lim_{s\to a^+}H_y(s)=\lim_{s\to b^-}H_y(s)=0
$$
and such that there are $\tilde{a}\le \tilde{b}\in (a,b)$ for which both $H_\theta(s)$ and $H_y(s)$ are monotonic in each of the intervals $(a,\tilde{a})$ and $(\tilde{b},b)$.
It follows from the assumptions on the function $H_\theta$ that it has either a global minimum or a global maximum, since $\lim_{s\to a^+}H_\theta(s)=\lim_{s\to b^-}H_\theta(s)$. 
At the corresponding point, the projection of the helix into the circle $y=0$ is singular, a \emph{fold point} of the helix.
Similarly, the function $H_y$ always has a global maximum, that will be called the \emph{maximum height} of the helix. See Figure \ref{homoclinic1}.

\begin{figure}[ht]
\begin{center}
\includegraphics[width=15cm]{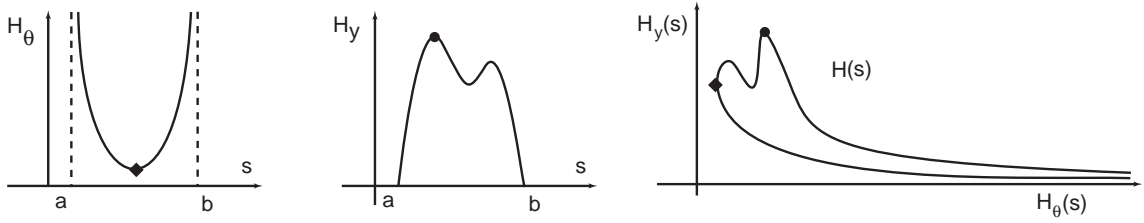}
\end{center}
\caption{\small A helix  is defined on a covering of the cylinder  by a smooth curve $(H_\theta(s),H_y(s))$ that turns around the cylinder infinitely many times as its height  $H_y$ tends do zero. It always contains a  fold point, indicated here by a diamond, and a point of maximum height, shown as a black dot.}
\label{homoclinic1}
\end{figure}

\begin{lemma}
\label{Structures} 
Consider a curve on one of the cylinder walls $\In(\vv)$ or $\Out(\ww)$, parametrised by the graph of a smooth function $h:[a,b]\rightarrow\RR$, where $b-a<2\pi$  with $h(a)=h(b)=0$, $h^\prime(a)>0$, $h^\prime(b)<0$ and $h(x)>0$ for all $x\in(a,b)$.
Let  $M$ be the global maximum value of $h$,  attained at a point $x_M\in(a,b)$.
Then, for the transition maps defined above, we have:
\begin{enumerate}
\item\label{helixOut}
 if the curve lies in $\In(\vv)$ then it is mapped by $\eta=\Phi _{\ww}\circ \Psi_{\vv \rightarrow \ww}\circ \Phi _{\vv}$ into  a helix on $\Out(\ww)$ accumulating on the circle  $\Out(\ww) \cap W^{u}_{\loc}(\ww)$,
 its maximum height is $M^\delta$ and
 it has a fold point at 
 $\eta(x_*,h(x_*))$ for some $x_*\in (a,x_M)$;
\item\label{helixIn}
 if the curve lies in $\Out(\ww)$ then  it is mapped by $\eta^{-1}$ into  a helix on $\In(\vv)$ accumulating on the circle  $\In(\vv) \cap W^{s}_{\loc}(\vv)$, its maximum height is $M^{1/\delta}$ and  it has a fold point at $\eta^{-1}(x_*,h(x_*))$ for some $x_*\in (x_M,b)$.
 \end{enumerate}
\end{lemma}

This result depends strongly on the form of $\eta$, and therefore, on the chirality of  the nodes.
If the nodes had different chirality, the graph of $h$ would no longer be mapped into a helix;
 the curve $\eta\left(x, h(x)\right)$ would instead have a vertical tangent at infinitely many points, see \cite{LR}.

\begin{proof}
The graph of $h$ defines a curve on $\In(\vv)$ without self-intersections. 
Since $\eta$ is the transition map of a differential equation, hence a diffeomorphism, this curve is mapped by $\eta$ 
into a curve 
$H (x)=\eta\left(x, h(x)\right)=\left(H _\theta(x),H _y(x)\right)$ in $\Out(\ww)$ 
without self-intersections.
Using the expression \eqref{eqeta} for $\eta$, we get
\begin{equation}\label{alphaPrime}
H (x)=\left(x-K\ln h (x),\left(h (x)\right)^\delta\right) \qquad \text{and}
\qquad
 H ^\prime(x)=\left(1-\frac{K h^\prime(x)}{h(x)}, \delta\frac{ \left(h (x)\right)^\delta h^\prime(x)}{h (x)}\right).
\end{equation}
From this expression it is immediate that $$\lim_{x\to a^+}H_\theta(x)=\lim_{x\to b^-}H_\theta(x)=+\infty \quad \text{and} \quad
\lim_{x\to a^+}H_y(x)=\lim_{x\to b^-}H _y(x)=0.$$
Also, $$H _y(x_M)=M^\delta \ge \left(h(x)\right)^\delta=H _y(x)$$ for all $x\in \left(a ,b\right)$, so the curve lies below the level $y=M^\delta$ in $\Out(\ww)$. Since $h^\prime(b)<0$, there is an interval $(\tilde{b},b)$ where $h^\prime(x)<0$ and hence $H_\theta^\prime(x)>0$ and $H_y^\prime(x)<0$. Similarly, $h^\prime(x)>0$ on some interval $(a,\hat{a})$, where $H_y^\prime(x)>0$.
For the sign of  $H _\theta^\prime(x)$, note that $K h^\prime(a )>0=h  (a )$ and 
 $K h^\prime(x_M)=0<h  (x_M)$.
 Thus, there is a point $x_*\in\left(a ,x_M\right)$ where $h (x_*)=Kh^\prime(x_*)$ and $H _\theta^\prime(x)$ changes sign, this is a fold point. If $x_*$ is the minimum value of $x$ for which this happens, then locally the helix lies to the right of this point. This proves assertion~\eqref{helixOut}. The proof of assertion~ \eqref{helixIn} is similar, using the expression 
 \begin{equation}\label{eqEtaInverse}
 \eta^{-1}(x,y)=\left(x+\frac{K}{\delta} \ln y , y^{1/\delta} \right).
 \end{equation}
 In this case we get $$\lim_{x\to a^+}H_\theta(x)=\lim_{x\to b^-}H_\theta(x)=-\infty$$ and if $x_*$ is the largest value of $x$ for which the helix has a fold point, then locally the helix lies to the left of $\eta^{-1}(x_*,h(x_*))$. 
 \end{proof} 
 
 Note that if, instead of the graph of a smooth function, we consider continuous, positive and piecewise smooth curve $\alpha(s)$ without self-intersections from $\alpha(0)=(a,0)$ to $\alpha(1)=(b,0)$,
 we can apply similar arguments to show that both $\eta(\alpha)$ and $\eta^{-1}(\alpha)$ are helices.

 \subsection{Geometry of the invariant manifolds}\label{subsecInvariantManifolds}
 
 There is also a well defined transition map
$$ 
\Psi_{\ww \rightarrow \vv}^\lambda:\Out(\ww)\longrightarrow \In(\vv)
$$  
that depends on the $\textbf{Z}_2\zg{2}$-symmetry breaking parameter $\lambda$, where $\Psi_{\ww \rightarrow \vv}^0$ is the identity map.
We will denote by  $R_\lambda$ the map  $\Psi_{\ww \rightarrow \vv}^\lambda \circ \eta$, where it is well defined.  When there is no risk of ambiguity, we omit the superscript  $\lambda$. In this section we investigate the effect of $\Psi_{\ww \rightarrow \vv}^\lambda$ on the two-dimensional invariant manifolds of $\vv$ and $\ww$ for  $\lambda\ne 0$, under the assumption  (P\ref{P5.}).

For this, let $f_\lambda$ be an unfolding of $f_0$ satisfying (P\ref{P1})--(P\ref{P5.}).
For $\lambda\ne 0$, we introduce the notation,  see Figure~\ref{elipse}:
\begin{itemize}
\item 
$(P_\ww^1,0)$ and $(P_\ww^2,0)$ with $0<P_\ww^1<P_\ww^2<2\pi$ are the coordinates of the two points where the connections  $[\ww \rightarrow \vv]$ of Property~(P\ref{P5.})   meet $\Out(\ww)$;
\item
$(P_\vv^1,0)$ and $(P_\vv^2,0)$  with $0<P_\vv^1<P_\vv^2<2\pi$  are the coordinates of the two  points where $[\ww \rightarrow \vv]$ meets $\In(\vv)$;
\item
$(P_\ww^j,0)$ and $(P_\vv^j,0)$ are on the same trajectory for each $j=1,2$.
\end{itemize}

 \begin{figure}[ht]
\begin{center}
\includegraphics[height=12cm]{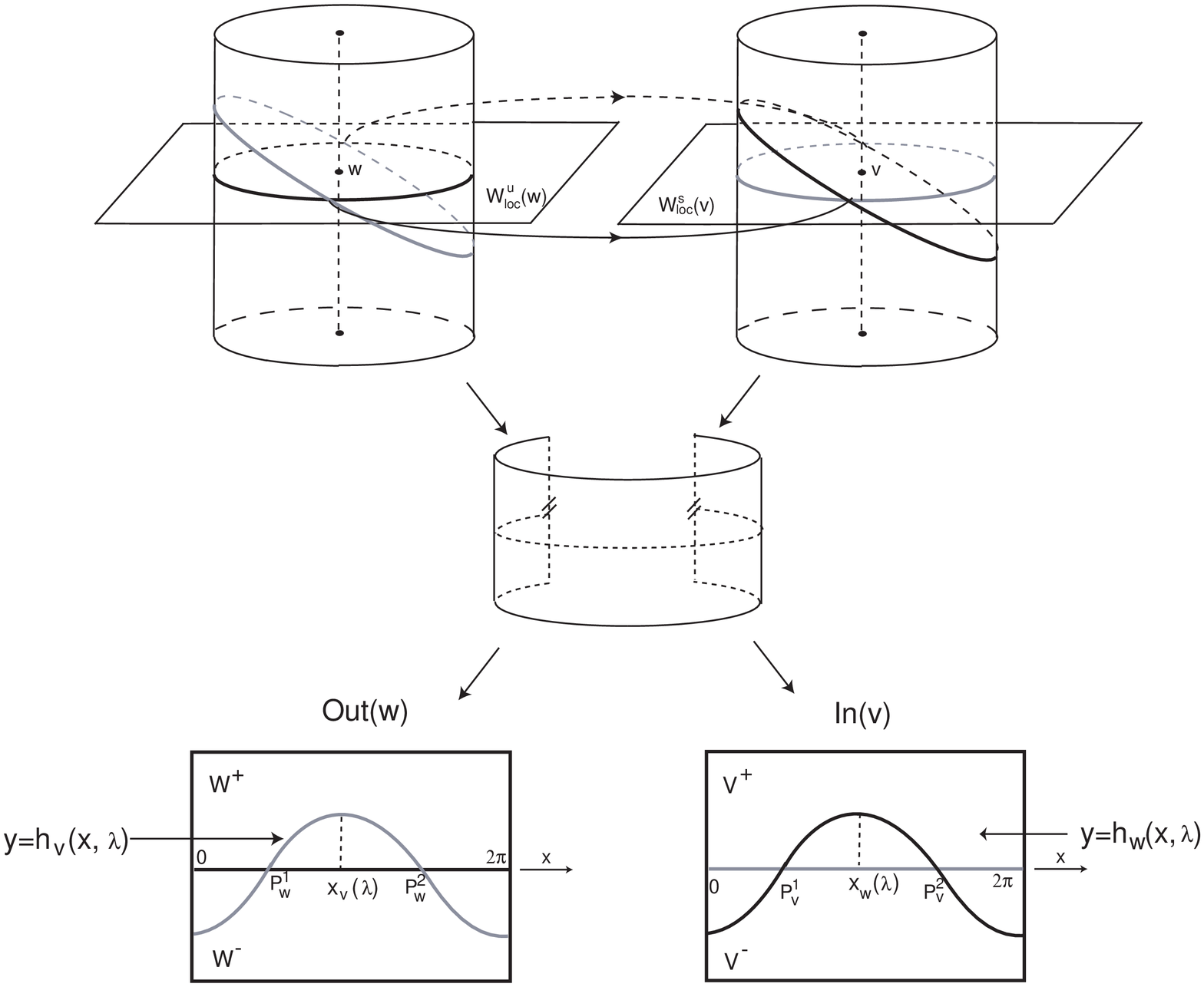}
\end{center}
\caption{\small For $\lambda$ close to zero, 
$W^s(\vv)$ intersects the wall $\Out(\ww)$ of the cylinder $W$ on a closed curve, given in local coordinates as the graph of a periodic function. Similarly, $W^u(\ww)$ meets $\In(\vv)$ on a closed curve
 --- this is the expected unfolding from the coincidence of the invariant manifolds at $\lambda=0$. }
\label{elipse}
\end{figure}

 By (P\ref{P5.}), for $\lambda \neq 0$, the manifolds $W^u(\ww)$ and $W^s(\vv)$ intersect transversely along the primary connections. For $\lambda$ close to zero, we are
assuming that $W^s_{\loc}(\vv)$ intersects the wall $\Out(\ww)$ of the cylinder $W$ on a closed curve as in Figure~\ref{elipse}.
It corresponds to the expected
unfolding from the coincidence of the manifolds $W^s(\vv)$ and $W^u(\ww)$ at $f_0$.
Similarly, $W^u_{\loc}(\ww)$ intersects the wall $\In(\vv)$ of the cylinder $V$ on a closed curve.
For small $\lambda>0$, these curves can be seen as graphs of smooth $2\pi$-periodic functions, for which we  make the following conventions:
\begin{itemize}
\item
$W^s_{\loc}(\vv)\cap \Out(\ww)$ 
is the graph of   $y=\wsv(x,\lambda)$, with $\wsv(P_\ww^j,\lambda)=0$, $j=1,2$;
\item
$W^u_{\loc}(\ww)\cap \In(\vv)$ is the graph
  of  $y=\wuw(x,\lambda)$, with $\wuw(P_\vv^j,\lambda)=0$, $j=1,2$;
 \item 
$\wsv(x,0)\equiv 0$ and $\wuw(x,0)\equiv 0$
\item for $\lambda>0$, we have
$\wsv ^\prime(P_\ww^1,\lambda)>0$, hence  
$  \wsv ^\prime(P_\ww^2,\lambda)<0$ and $\wuw ^\prime(P_\vv^1,\lambda)<0,\wuw ^\prime(P_\vv^2,\lambda)>0$.
\end{itemize}

The two points $(P_\ww^1,0)$ and $(P_\ww^2,0)$ divide the closed curve $y=\wsv(x,\lambda)$
  in two components, corresponding to different signs of the second coordinate.
 With the conventions  above, we get $\wsv (x,\lambda)>0$ for  $x\in\left(P_\ww^1,P_\ww^2\right)$.
Then the region $W^-$ in $\Out(\ww)$ delimited by  $W_{\loc}^s(\vv)$ and  $W_{\loc}^u(\ww)$ between $P_\ww^1$ and $P_\ww^2$ gets mapped by $ \Psi_{\ww \rightarrow \vv}^\lambda$ into $\In^-(\vv)$, while all other points  in  $\Out^+(\ww)$  are mapped into  $\In^+(\vv)$. 
We denote by $W^+$ the latter set,  of points in $\Out(\ww)$ with $0<y<1$ and $y>\wsv(x,\lambda)$ for $x\in\left(P_\ww^1,P_\ww^2\right)$.
The maximum value of $\wsv(x,\lambda)$ is  attained at some point 
$$
(x,y)= (\xmsv(\lambda),\ymsv(\lambda)) \qquad \text{with} \qquad P_\ww^1<\xmsv(\lambda)<P_\ww^2.
$$

Finally, let  $\ymuw(\lambda)$ be the maximum of $\wuw(x,\lambda)$, attained at a point 
 $\xmuw(\lambda)\in \left(P_\vv^2,P_\vv^1\right)$.  
 In order to simplify the writting, the following analysis is focused on the two Bykov cycles whose connection $[\vv \to \ww]$ lies in the subspace defined by $ x_3\geq0$. 
The same results, with minimal adaptations, hold for cycles obtained from the other connection.
 With this notation, we have:

\begin{proposition}\label{PropWuw}
Let $f_\lambda$ be family of vector fields satisfying (P\ref{P1})--(P\ref{P5.}).
For $\lambda\ne 0$ sufficiently small, the portion of $W^u_{\loc}(\ww)\cap \In(\vv)$ that lies in $\In^+(\vv)$
is mapped by $\eta$ into a helix in $\Out(\ww)$ accumulating on $W^u_{\loc}(\ww)$.
If $\ymuw(\lambda)$ is the maximum height of $W^u_{\loc}(\ww)\cap \In^+(\vv)$, then the maximum height of the helix is $\ymuw(\lambda)^\delta$. 
For each $\lambda>0$ there is a fold point in the helix that, as $\lambda$ tends to zero, turns around the cylinder wall $\Out(\ww)$  infinitely many times.
\end{proposition}


\begin{proof}
That $\eta$ maps $W^u_{\loc}(\ww)\cap \In^+(\vv)$ into a helix, and the statement about its maximum height follow directly by applying assertion~\eqref{helixOut} of Lemma~\ref{Structures} to $\wuw$.  
For the  fold point in the helix, let $x_*(\lambda)$ be its first coordinate. From the expression \eqref{eqeta} of 
$\eta$ it follows that $$x_*(\lambda)=x_\lambda-K\ln \wuw(x_\lambda)$$ for some 
$x_\lambda\in(P_\vv^2, x_\ww(\lambda))$ and with 
$\wuw(x_\lambda, \lambda)\le \wuw(x_\ww(\lambda),\lambda)=M_\ww(\lambda)$ and hence, 
$$-K\ln \wuw(x_\lambda,\lambda)\ge -K\ln M_\ww(\lambda).$$  
Since $f_\lambda$ unfolds $f_0$, then $\lim_{\lambda\to 0}M_\ww(\lambda)=0$, hence
$\lim_{\lambda\to 0}-K\ln \wuw(x_\lambda,\lambda)=\infty$ and therefore, 
 the fold point  turns around the cylinder $\Out(\ww)$  infinitely many times.
\end{proof}

The Hypothesis (P\ref{P6}) about chirality of the nodes
 is essential for
 Lemma~\ref{Structures} and Proposition~\ref{PropWuw}. 
If  we had taken the rotation in $W$ with the opposite orientation to that in $V$, the rotations would cancel out and the curve $H(x)$ defined in (\ref{alphaPrime}) would no longer  be a helix, since the angular coordinate $H_\theta$ would not be monotonic. 

\section{Heteroclinic Tangencies}
\label{sec tangency}

Using the notation and results of Section~\ref{localdyn} we can now discuss the tangencies of the invariant manifolds and prove Theorems~\ref{teorema tangency} and \ref{teorema multitangency}.
As remarked in  Section~\ref{subsecInvariantManifolds}, since $f_\lambda$ unfolds $f_0$, then the maximum heights, $\ymsv(\lambda)$ of $W^s_{\loc}(\vv)\cap \Out(\ww)$, and $\ymuw(\lambda)$ of $W^u_{\loc}(\ww)\cap \In(\vv)$, satisfy:
$$
\lim_{\lambda\to 0}\ymsv(\lambda)=\lim_{\lambda\to 0}\ymuw(\lambda)=0.
$$
We make the additional assumption  that  $\left(\ymuw(\lambda)\right)^\delta$ tends to zero faster than $\ymsv(\lambda)$. This condition defines the open set  ${\mathcal C}$
of unfoldings $f_\lambda$ that we need for the statement of Theorem~\ref{teorema tangency}. 
%

The subset ${\mathcal C}$  of Theorem~\ref{teorema tangency} is the set of  families $f_\lambda$ of vector fields satisfying (P\ref{P1})--(P\ref{P5.}) for which 
  there is a value $\lambda_*>0$ such that  for $0<\lambda<\lambda_*$ we have $\left(\ymuw(\lambda)\right)^\delta<\ymsv(\lambda)$.
  Then  ${\mathcal C}$ is an open subset, in the $C^3$ topology,  of the set of families $f_\lambda$ of vector fields satisfying (P\ref{P1})--(P\ref{P5.}).

\begin{figure}[ht]
\begin{center}
\includegraphics[width=14cm]{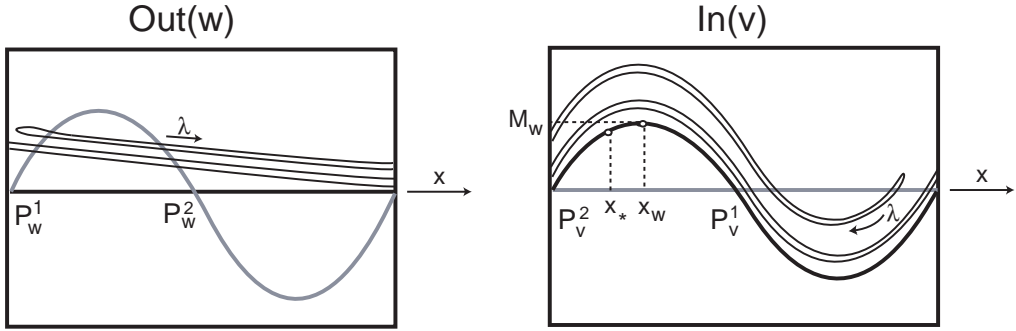}
\end{center}
\caption{\small Left: when $\lambda$ decreases, the fold point of the helix $\alpha^\lambda(x)\in \Out(\ww)$ moves to the right and  for  $\lambda=\lambda_i$ it is tangent to $W^s_{\loc}(\vv)$ creating a 1-pulse tangency. 
Right:  $\Psi_{[\ww\rightarrow\vv]}$ maps the helix $\alpha^\lambda(x)$ close to $W^u_{\loc}(\ww)\cap \In(\vv)$, creating several curves that satisfy the hypotheses of Lemma~\ref{Structures} \eqref{helixOut}. These curves are again mapped by $\eta$ into helices in $\Out(\ww)$ creating 2-pulse heteroclinic tangencies.}
\label{homoclinic2}
\end{figure}

\subsection{Proof of Theorem~\ref{teorema tangency}}

Suppose $f_\lambda\in{\mathcal C}$.
By Proposition~\ref{PropWuw}, the curve 
$$\alpha^\lambda(x)=\eta\left(x,\wuw(x,\lambda),\lambda\right)=\left(\alpha_1^\lambda(x),\alpha_2^\lambda(x)\right), \qquad x\in\left(P_\vv^2,P_\vv^1\right)\pmod{2\pi}$$
is a helix in $\Out^+(\ww)$ and has at least one fold point at $x=x_*(\lambda)$.
The second coordinate of the helix satisfies $0<\alpha_2^\lambda(x)<\ymuw(\lambda)^\delta$ for all $x\in\left(P_\vv^2,P_\vv^1\right)\pmod{2\pi}$ and all positive $\lambda<\lambda_*$.
Since $f_\lambda\in{\mathcal C}$, then $\alpha_2^\lambda(x)<\ymsv(\lambda)$ for all $x$ and  all positive $\lambda<\lambda_*$.

Moreover, since the fold point $\alpha^\lambda(x_*(\lambda))$ turns around $\Out(\ww)$ infinitely many times as $\lambda$ goes to zero,
given any $\lambda_0<\lambda_*$
there   exists a positive value $\lambda_R<\lambda_0$ such that 
$\alpha^\lambda(x_*(\lambda_R))$
lies in $W^+$, the region in $\Out(\ww)$ between $W_{\loc}^s(\vv)$ and  $W_{\loc}^u(\ww)$ 
that gets mapped into the upper part of $\In(\vv)$. 
Since the second coordinate of the 
fold point is less than the maximum of $\wsv$,
 there is   a positive value $\lambda_L<\lambda_R$ such that $\alpha^\lambda(x_*(\lambda_L))$ lies in $W^-$,
  the region in $\Out(\ww)$  that gets mapped into the lower part of $\In(\vv)$, whose boundary  contains the graph of $\wsv$. 
Therefore,  the curve $\alpha^\lambda(x_*(\lambda))$ is tangent to the graph of $\wsv(x,\lambda)$ at some point  $\alpha(x_*(\lambda_1))$ with $\lambda_1\in\left( \lambda_L,\lambda_R\right)$.

We have thus shown  that given $\lambda_0>0$, there is  some positive $\lambda_1<\lambda_0$,
for which the image of the curve $W_{\loc}^u(\ww)\cap \In(\vv)$ by  $\eta$ is tangent to $W_{\loc}^s(\vv)\cap \Out(\ww)$, creating a 1-pulse heteroclinic tangency. 
Two transverse 1-pulse  heteroclinic connections exist for $\lambda>\lambda_1$ close to $\lambda_1$. These connections come together at the tangency and disappear.

By Proposition \ref{PropWuw}, as  $\lambda$ goes to zero, the fold point $\alpha^\lambda(x_*(\lambda))$  turns around the cylinder $\Out(\ww)$  infinitely many times, thus going in and out of $W^-$.
Each times it crosses the boundary, a new tangency occurs.
Repeating the argument above yields the sequence $\lambda_i$ of  parameter values for which there is a 1-pulse heteroclinic tangency and this completes the proof of the main statement of Theorem~\ref{teorema tangency}.

On the other hand,  as $\lambda$ goes to zero,  the maximum height of the  helix, $\left(\ymuw(\lambda)\right)^\delta$ also tends to zero. This implies that the second coordinate of the points 
$\alpha^\lambda(x_*(\lambda_i))\in \Out(\ww)$ where there is a  1-pulse  heteroclinic tangency tends to zero as $i$ goes to infinity. This shows that the tangency approaches the two-dimensional  connection $[\ww\to\vv]$ that exists for $\lambda=0$.
\Qed

A crucial fact in the proof of Theorem~\ref{teorema tangency} is that $W^u_{\loc}(\ww)\cap \In(\vv)$ is the graph of a function with a single maximum,  and that this maximum goes to 0 as $\lambda$ goes to 0.
This follows because the family  $f_\lambda$  unfolds the more symmetric vector field $f_0$.
We could make the assumption on the invariant manifold directly, but it is  the context of  symmetry-breaking that makes it natural.

The construction in the proof of Theorem~\ref{teorema tangency} may be extended to obtain multipulse tangencies, as follows:

\subsection{Proof of Theorem~\ref{teorema multitangency}}
Look at $W^u_{\loc}(\ww)\cap \In(\vv)$, the graph of $\wuw(x,\lambda)$.
Since $\wuw ' (P_\vv^1,\lambda)<0$, there is an interval $[\tilde{x}, P_\vv^1]\subset  [\xmuw(\lambda), P_\vv^1]$ where the map $\wuw$ is monotonically decreasing. 
 Note that the hypothesis of different chirality is essential here.
Therefore, we may define infinitely many intervals where $\alpha^\lambda(x)=\eta\left(x,\wuw(x,\lambda)\right)$ lies in $W^+$. 
More precisely, we have two sequences, $(a_j)$ and $(b_j)$ in $[\tilde{x}, P_\vv^1]$ such that:
\begin{itemize}
\item
$a_j<b_j<a_{j+1}$\quad with \quad $\lim_{j\to\infty} a_j=P_\vv^1$;
\item
$\alpha^\lambda(a_j)$ and $\alpha^\lambda(b_j)\in W^s_{loc}(\vv)\cap \Out^+(\ww)$;
\item
if $x\in(a_j,b_j)$ then $\alpha^\lambda(x)\in W^+$;
\item
the curves $\alpha^\lambda\left([a_j,b_j]\right)$ accumulate uniformly on $W^u_{\loc}(\ww)\cap \Out(\ww)$ as $j\to\infty$.
\end{itemize}
Hence, each one of the curves $\alpha^\lambda\left([a_j,b_j]\right)$ is mapped by $\Psi_{\ww \rightarrow \vv}$ into the graph of a function $h$ in $\In(\vv)$ satisfying the conditions of Lemma~\ref{Structures},
and hence each one of these curves  is mapped by $\eta$ into a  helix $\xi_j(x)$, $x\in(a_j,b_j)$.
 
 \begin{figure}[ht]
\begin{center}
\includegraphics[scale=.6]{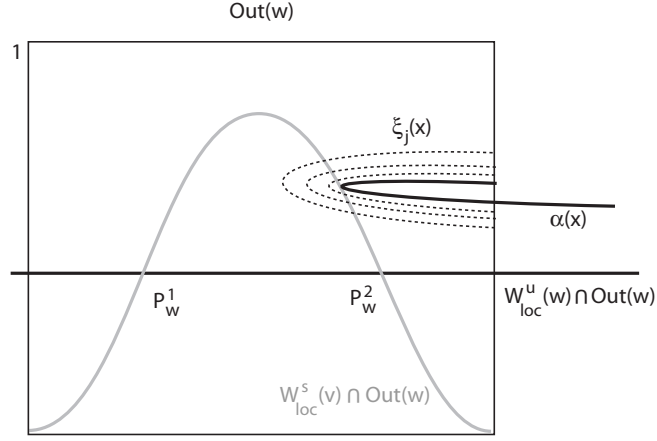}
\end{center}
\caption{\small Curves in the proof of Theorem~\ref{teorema multitangency}:  for $\lambda=\lambda_i$ the curve $\alpha^\lambda(x)=\eta(W^u_{\loc}(\ww)\cap \In^+(\vv))$ (solid black curve) is tangent to $W^s_{\loc}(\vv)$ (gray curve) in $\Out(\ww)$.
The curves $\xi_j(x)$ (dotted) accumulate on $\alpha^\lambda(x)$. Very small changes in $\lambda$ make them tangent to $W^s_{\loc}(\vv)$ creating 2-pulse heteroclinic tangencies.} 
\label{acumula}
\end{figure}

Let $\lambda_i$ be a parameter value for which  $\dot{x}=f_{\lambda_i}(x)$ has a 1-pulse heteroclinic tangency as stated in Theorem~\ref{teorema tangency}. 
As $j \rightarrow + \infty$, the helices $\xi_j(x)$ accumulate on the helix of Theorem~\ref{teorema tangency} as drawn in Figure~\ref{acumula}, hence the fold point of $\xi_j(x)$ is arbitrarily close to the  fold point of $\eta(x,\wuw(x,\lambda))$. 
 The arguments in the proof of Theorem~\ref{teorema tangency} show that  a small change in the parameter $\lambda$ makes the new helix tangent to $W^s_{\loc}(\vv)$ as in Figure~\ref{homoclinic2}.
 For each $j$ this creates  a 2-pulse heteroclinic tangency  at $\lambda=\lambda_{ij}$.
 Since the the helices $\xi_j(x)$ accumulate on $\alpha^{\lambda_i}(x)$, it follows that
 $\lim_{j \in \NN} \lambda_{ij}=\lambda_i$.
 
 Finally, the argument may be applied  recursively to show that each $n$-pulse heteroclinic tangency is accumulated by $(n+1)$-pulse heteroclinic tangencies for nearby parameter values.

\Qed

If $\lambda^\star \in \RR$ is such that the flow of $\dot{x}=f_{\lambda^\star}(x)$ has a heteroclinic tangency, when $\lambda$ varies near $\lambda^\star$, we find the creation and the destruction of horseshoes that is accompanied by Newhouse phenomena \cite{Gonchenko2007, YA}. In terms of numerics, we know very little about the geometry of these  attractors, we also do not know the size and the shape of their basins of attraction. Basin boundary metamorphoses and explosions of chaotic saddles as those described in \cite{RAOY} are expected.

\section{Bifurcating dynamics}
\label{bif}

We discuss here the geometric constructions that determine the global dynamics near a Bykov cycle, in order to prove Theorem~\ref{newhouse}.
For this we need some preliminary definitions and more information on the geometry of the transition maps. 
First, we adapt the definition of horizontal strip in \cite{GH} to serve our purposes:
for $\tau>0$ sufficiently small, in the local coordinates of the walls of the cylinders $V$ and $W$, consider the rectangles:
$$
\quadr_\vv=  [P_\vv^2-\tau, P_\vv^1+\tau] \times [0, 1]\subset \In(\vv)
\qquad \text{and} \qquad \quadr_\ww=  [P_\ww^2-\tau, P_\ww^1+\tau] \times [0, 1] \subset \Out(\ww)
$$
with the conventions  $-\pi<P_\vv^2-\tau< P_\vv^1+\tau\le\pi$ and
$-\pi<P_\ww^2-\tau<P_\ww^1+\tau\le \pi$.

A \emph{horizontal strip} in $\quadr_\vv$ is a subset of  $\In^+(\vv)$ of the form
 $$
 \hor=\left\{(x,y)\in In^+(\vv):\quad  x \in [P_\vv^2-\tau, P_\vv^1+\tau]\qquad \text{and} \quad y \in [u_1(x), u_2(x)]\right\},
 $$
where $u_1, u_2: [P_\vv^2-\tau, P_\vv^1+\tau] \rightarrow (0,1]$ are smooth maps such that $u_1(x)<u_2(x)$ for all $x \in [P_\vv^2-\tau, P_\vv^1+\tau]$. The graphs of  the $u_j$ are called the \emph{horizontal boundaries} of $\hor$
and the segments $\left(P_\vv^2-\tau,y\right)$, $u_1(P_\vv^2-\tau)\le y \le u_2(P_\vv^2-\tau)$ and 
$\left(P_\vv^1+\tau,y\right)$, with $u_1(P_\vv^1+\tau)\le y \le u_2(P_\vv^1+\tau)$ are its \emph{vertical boundaries}. Horizontal and vertical boundaries intersect at four \emph{vertices}.
The  \emph{maximum height} and the  \emph{minimum height}   of $\hor$ are, respectively
$$
\max_{x \in [P_\vv^2-\tau, P_\vv^1+\tau]}u_2(x)
\qquad\qquad
\min_{x \in [P_\vv^2-\tau, P_\vv^1+\tau]}u_1(x).
$$
Analogously  we may define  a \emph{horizontal strip} in $\quadr_\ww \subset \Out^+(\ww)$. 

A \emph{horseshoe strip} in $\quadr_\vv$ is a subset of $\In^+(\vv)$ of the form
$$
\left\{ (x,y)\in \In^+(\vv):\quad x\in[a_2,b_2], \quad y \in [u_1(x), u_2(x)]\right\},
 $$
 where
 \begin{itemize}
\item $[a_1, b_1] \subset [a_2, b_2] \subset [P_\vv^2-\tau, P_\vv^1+\tau]$;
\item $u_1(a_1)=u_1(b_1)=0$;
\item $u_2(a_2)=u_2(b_2)=0$.
\end{itemize}
The boundary of $\hor$ consists of the graph of $u_2(x)$, $x\in[a_2,b_2]$, the graph of $u_1(x)$,  $x\in[a_1,b_1]$, together with the two pieces of $W^s_{\loc}(\vv)\cap \In(\vv)$ with $x\in[a_2,a_1]$ and $x\in[b_1,b_2]$.

\subsection{Horseshoe strips in $\quadr_\vv$}

We are interested in the dynamics of points whose trajectories start in $\quadr_\vv$ and return to $\In^+(\vv)$ arriving at $\quadr_\vv$.

\begin{figure}[ht]
\begin{center}
\includegraphics[scale=.6]{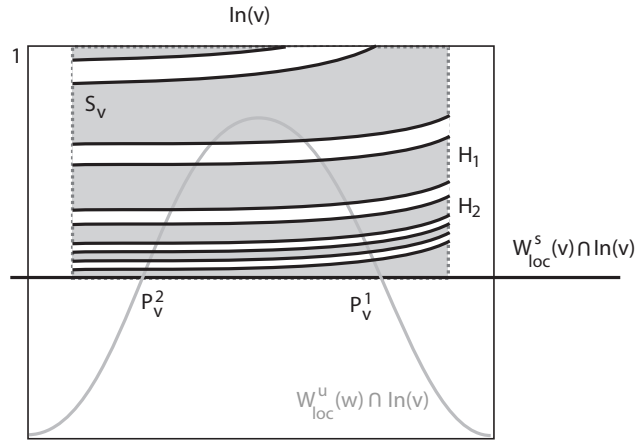}
\end{center}
\caption{\small For $\lambda>0$ sufficiently small,  the set $\eta^{-1}(\quadr_\ww)\cap \quadr_\vv$ has infinitely many connected components ({gray region}), all   of which, except maybe for the top ones, define horizontal strips  in $\quadr_\vv$ accumulating on $W^s_{loc}(\vv)\cap \In(\vv)$ ({gray} curve).}
\label{strips}
\end{figure}

\begin{lemma}\label{lemaHorizontalStrips}
The set $\eta^{-1}(\quadr_\ww)\cap \quadr_\vv$ has infinitely many connected components all of which are horizontal strips in  $\quadr_\vv$ accumulating on $W^s_{\loc}(\vv)\cap \In(\vv)$, except maybe for a finite number. The horizontal boundaries of these strips are graphs of monotonically increasing functions of $x$. 
\end{lemma}

\begin{proof}
The boundary of $\quadr_\ww$ consists of the following:
\begin{enumerate}
\item \label{baixo}
a piece of $W^u_{\loc}(\ww)\cap \Out(\ww)$ parametrised by $y=0$, $x\in[P^1_\ww-\tau, P^2_\ww+\tau]$, where $\eta^{-1}$ is not defined;
\item \label{cima}
the horizontal segment $(x,1)$ with $x\in[P_\ww^2-\tau, P_\ww^1+\tau]$;
\item\label{lados}
two vertical segments $\left(P^1_\ww-\tau, y\right)$ and $\left(P^2_\ww+\tau,y\right)$ with $y\in\left(0,1\right)$.
\end{enumerate}
Together, the components \eqref{cima} and \eqref{lados} form a continuous curve that, by the arguments of Lemma~\ref{Structures} \eqref{helixIn}, is mapped by $\eta$ into a helix on $\In^+(\vv)$, accumulating on $W^s_{\loc}(\vv)\cap \In(\vv)$.
As the helix approaches $W^s_{\loc}(\vv)$, it crosses the vertical boundaries of $\quadr_\vv$ infinitely many times.
The interior of $\quadr_\ww$ is mapped into the space between consecutive crossings, intersecting $\quadr_\vv$ in horizontal strips, as shown in Figure~\ref{strips}.
From the expression \eqref{eqEtaInverse} of $\eta^{-1}$ it also follows that the vertical boundaries of $\quadr_\ww$ are mapped into graphs of monotonically increasing functions of $x$.
\end{proof}

 Here, as in the next  two Lemmas, we are using the form of the map $\eta$ through  Lemma~\ref{Structures}, hence the result depends strongly on the chirality hypothesis (P\ref{P6}).

Denote by $\hor_n$ the strip that attains its maximum height $h_n$  at the vertex $\left(P^1_\vv+\tau,h_n \right)$ with
\begin{equation}\label{maxHn}
h_n=\ee^{(P^1_\vv-P^2_\ww+2\tau-2n\pi)/K} ,
\end{equation}
 then 
 $\lim_{n\to\infty}h_n=0$, hence the strips  $\hor_n$ accumulate on  $W^{s}_{\loc}(\vv)\cap \In(\vv)$.
 The minimum height of $\hor_n$ is given by
 \begin{equation}\label{minHn}
 m_n=\ee^{(P^2_\vv-P^1_\ww-2\tau-2n\pi)/K}
 \end{equation}
and is attained at the vertex 
 $\left(P^2_\vv-\tau,m_n\right)$.
 Moreover $n<m$ implies that $\hor_n$ lies above $\hor_m$.

\begin{figure}[ht]
\begin{center}
\includegraphics[height=7.6cm]{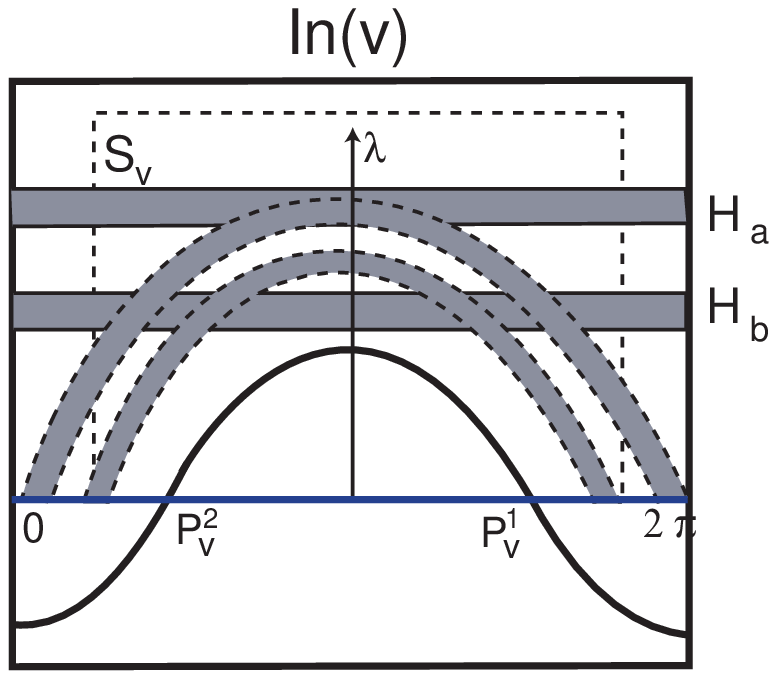}
\end{center}
\caption{\small Horseshoe strips: for $\lambda>0$ sufficiently small, the set $R_\lambda(\eta^{-1}(\quadr_\ww)\cap \quadr_\vv)  \subset \In^+(\vv)$ defines a countably infinite number of horseshoe strips in $\quadr_\vv$, that accumulate on $W^u_{\loc}(\ww)\cap \In(\vv)$.}
\label{transition1}
\end{figure}

\begin{lemma}\label{lemaEtaHorizontalStrips}
Let $\hor_n$ be one of the   horizontal strips in  $\eta^{-1}(\quadr_\ww)\cap \quadr_\vv$.
Then $\eta(\hor_n)$ is a horizontal strip in $\quadr_\ww$.
The strips  $\eta(\hor_n)$ accumulate on $W^u_{\loc}(\ww)\cap \In(\vv)$ as $n\to\infty$
and the maximum height of $\eta(\hor_n)$ is $h_n^\delta$, where $h_n$ is the maximum height of 
$\hor_n$.
\end{lemma}

\begin{proof}
The boundary of $\quadr_\vv$ consists of a piece of $W^s_{\loc}(\vv)$ plus a curve formed by three segments, two of which are vertical and a horizontal one.
From the arguments of Lemma~\ref{Structures} \eqref{helixOut}, it follows that the part of the boundary of $\quadr_\vv$ not contained in $W^s_{\loc}(\vv)$ is mapped by $\eta$ into a helix.

Consider now the effect of $\eta$ on the boundary of $\hor_n$.
Each horizontal boundary gets mapped into a piece of one of the vertical boundaries of $\quadr_\ww$.
The vertical boundaries of $\hor_n$ are contained in those of $\quadr_\vv$ and hence are mapped into two pieces of a helix, that will form the horizontal boundaries of the strip $\eta(\hor_n)$,
that may be written as graphs of decreasing functions of $x$.

A shown after Lemma~\ref{lemaHorizontalStrips}, the  maximum height of $\hor_n$ tends to zero, hence the strips $\eta(\hor_n)$ have the same property.
The maximum height  of $\eta(\hor_n)$ is $h_n^\delta$, attained at the point
 $\left(P^2_\ww-\tau,h_n^\delta\right)$.
\end{proof}

\begin{lemma}\label{lemaHorseshoeStrips}
For  each $\lambda>0$ sufficiently small, there exists $n_0(\lambda)$ such that for all $n\ge n_0$ the image $R_\lambda(\hor_n)$ of the   horizontal strips in  $\eta^{-1}(\quadr_\ww)\cap \quadr_\vv$
intersects $\quadr_\vv$ in a horseshoe strip. 
The strips $R_\lambda(\hor_n)$ accumulate on $W^u_{\loc}(\ww)\cap \In(\vv)$ and, when  $n\to\infty$, their maximum height tends to $\ymuw(\lambda)$,  the maximum height of  $W^u_{\loc}(\ww)\cap \In(\vv)$.
\end{lemma}

\begin{proof}
The curve $W^s_{\loc}(\vv)\cap \Out(\ww)$ is the graph of the function $\wsv(x,\lambda)$ that is positive for $x$ outside the interval $[P^2_\ww,P^1_\ww]\pmod{2\pi}$.
In particular, $\wsv(P^2_\ww-\tau,\lambda)>0$ and $\wsv(P^2_\ww+\tau,\lambda)>0$ for small 
$\tau>0$.
Therefore there is a piece of the vertical boundary of $\quadr_\ww$ that lies below $W^s_{\loc}(\vv)\cap \Out(\ww)$, consisting of the two segments $(P^2_\ww-\tau,y)$ with $0<y<\wsv(P^2_\ww-\tau,\lambda)$ and $(P^1_\ww+\tau,y)$ with $0<y<\wsv(P^1_\ww+\tau,\lambda)$.
For small $\lambda>0$, these segments are mapped by $\Psi_{\ww\to\vv}^\lambda$ inside $\In^-(\vv)$.

Let $n_0$ be such that the maximum height of $\eta(\hor_{n_0})$ is less than the minimum of 
$$\wsv(P^2_\ww-\tau,\lambda)>0 \qquad \text{and} \qquad \wsv(P^2_\ww+\tau,\lambda)>0.$$
Then for any $n\ge n_0$ the vertical sides of $\eta(\hor_{n})$ are mapped  by $\Psi_{\ww\to\vv}^\lambda$ inside $\In^-(\vv)$.
The horizontal boundaries of $\eta(\hor_n)$ go across $\quadr_\ww$, so writing them  them as graphs of $u_1(x)<u_2(x)$, there is an interval where the second coordinate of $\Psi_{\ww\to\vv}^\lambda(x,u_j(x))$ is more than $\ymuw(\lambda)>0$, the maximum height of $W^u_{\loc}(\ww)\cap \In^+(\vv)$. 
Since the second coordinate of $\Psi_{\ww\to\vv}^\lambda(x,u_j(x))$ changes sign  twice, then it equals zero at two points, hence $R_\lambda(\hor_n)\cap \quadr_\vv$ is a horseshoe strip.

We have shown that the maximum height of $\eta(\hor_n)$ tends to zero as $n\to\infty$, hence the maximum height of $R_\lambda(\hor_n)=\Psi_{\ww\to\vv}^\lambda(\eta(\hor_n))$ tends to 
$\ymuw(\lambda)$.
\end{proof}

\subsection{Regular Intersections of Strips}
We now discuss  the global dynamics near the Bykov cycle.
The structure of the non-wandering set near the network  depends on the geometric properties of the intersection of $\hor_n$ and $R_\lambda(\hor_n)$.

Let $A$ be a  horseshoe strip and $B$ be a horizontal  strip in $\quadr_\vv$. 
We say that $A$ and $B$ \emph{intersect regularly} if $A\cap B \neq \emptyset$ and each one of the horizontal boundaries of $A$ goes across
each one of   the horizontal boundaries of $B$. 
Intersections that are  neither empty nor regular, will be called  \emph{irregular}.

If the  horseshoe strip $A$ and the  horizontal  strip $B$ intersect regularly, then $A\cap B$ has at least two connected components, see Figure \ref{transition1}.
In this and the next subsection, we will find that the horizontal strips $\hor_n$ across $\quadr_\vv$ may intersect $R_\lambda(\hor_m)$ in the three ways: empty, regular and irregular, but there is an ordering for the type of intersection,  as shown in Figure \ref{new2}. 

\begin{lemma}
\label{Novo2}
For any given fixed $\lambda>0$ sufficiently small, there exists $N(\lambda) \in \NN$ such that for all $m,n>N(\lambda)$, the  horseshoe strips $R_\lambda(\hor_m)$ in $\quadr_\vv$  intersect each one of the horizontal strips $\hor_n$ regularly.
\end{lemma}

\begin{proof}
From Lemma~\ref{lemaHorseshoeStrips} we obtain $n_0(\lambda)$ such that all $R_\lambda(\hor_m)$ with $m\ge n_0(\lambda)$ are horseshoe strips and their lower horizontal boundary has maximum height bigger than  the maximum height  $\ymuw(\lambda)$ of $W^u_{\loc}(\ww)\cap \In(\vv)$.

On the other hand, since the strips $\hor_n$ accumulate uniformly on $W^s_{\loc}(\vv)\cap \In(\vv)$, with their maximum height $h_n$ tending to zero, then there exists $n_1(\lambda)$ such that 
$h_n<\ymuw(\lambda)$ for all $n\ge n_1(\lambda)$. Note that  the $h_n$ do not depend on $\lambda$. 
Therefore, for $m,n>N(\lambda)=\max\{n_0(\lambda),n_1(\lambda)\}$ both horizontal boundaries of $R_\lambda(\hor_m)$ go across the two horizontal boundaries of $\hor_n$.
\end{proof}

The constructions of this section also hold for the backwards return map $R^{-1}$ with analogues  to Lemmas~\ref{lemaHorizontalStrips}, \ref{lemaEtaHorizontalStrips}, \ref{lemaHorseshoeStrips} and \ref{Novo2}.

Generically, for $n>N(\lambda)$ each horizontal strip $\hor_n$ intersects $R_\lambda(\hor_n)$ in two connected components. Thus the dynamics of  points whose trajectories always return to $\In(\vv)$ in $\hor_n$ may be coded by a full shift on two symbols, that describe which component is visited by the trajectory on each return to $\hor_n$.
Similarly, trajectories that return to $\quadr_\vv$ inside $\hor_n\cup\cdots\cup\hor_{n+k}$ may be coded by a full shift on $2k$ symbols. As $k\to\infty$, the strips $\hor_{n+k}$ approach $W^s_{\loc}(\vv)\cap \In(\vv)$ and the number of symbols tends to infinity.
We have recovered  the horseshoe dynamics described in assertion~\eqref{item5} of Theorem~\ref{teorema T-point switching}.

The regular intersection of Lemma~\ref{Novo2} implies the existence of an $R_\lambda$-invariant subset in $\eta^{-1}(\quadr_\ww)\cap \quadr_\vv$, 
the Cantor set of initial conditions:
$$
\Lambda=\bigcap_{j\in\ZZ} \bigcup_{m,n\ge N(\lambda)} \left({R_\lambda}^j(\hor_m)\cap\hor_n\right),
$$
 where the return map to $\eta^{-1}(\quadr_\ww)\cap \quadr_\vv$ is well defined in forward and backward time, for arbitrarily large times. 
We have shown here that the map $R_\lambda$ restricted to this set is semi-conjugate to a full shift over a countable alphabet.
Results of \cite{ACL NONLINEARITY,LR} show that the first return map is hyperbolic in each horizontal strip, implying the full conjugacy to a shift. The time of return of points in $\hor_j$ tends to $+\infty$, as $j \to +\infty$. 
The set $\Lambda$ depends strongly on the parameter $\lambda$,
 in the next subsection we discuss its bifurcations when $\lambda$ decreases to zero.

\subsection{Irregular intersections of strips}
The horizontal strips $\hor_n$ that comprise $\eta^{-1}(\quadr_\ww)\cap\quadr_\vv$ do not depend on the bifurcation parameter $\lambda$, as shown in Lemmas~\ref{lemaHorizontalStrips} and \ref{lemaEtaHorizontalStrips}.
This is in contrast with the strong dependence on $\lambda$ shown by the first return of these points to $\quadr_\vv$ at the horseshoe strips $R_\lambda(\hor_n)$.
In particular, the values of $n_0(\lambda)$ (Lemma~\ref{lemaHorseshoeStrips}) and $N(\lambda)$ (Lemma~\ref{Novo2}) vary with the choice of $\lambda$.
For  a small fixed $\lambda>0$ and for $m,n\ge N(\lambda)$ we have shown that $\hor_n$ and 
$R_\lambda(\hor_m)$ intersect regularly.

The next result describes the bifurcations of these sets when $\lambda$ decreases.
These global bifurcations have been  described by Palis and Takens in \cite{PT} in a different context, where the horseshoe strips are translated down as a parameter varies.
In our case,  when $\lambda$ goes to zero the horseshoe strips  are flattened into the common  invariant two-dimensional manifoldsof $\vv$ and $\ww$.

\begin{proposition}
\label{regular_intersections}
Given $\lambda_3>0$ sufficiently small, 
there exist $\lambda_1< \lambda_2< \lambda_3 \in \RR^+$,
a horizontal strip  $\hor_a $ across $\quadr_\vv \subset \In^+(\vv)$ and  $b_0>a$ such that for any $b>b_0$ the horizontal strips $\hor_a $ and $\hor_b$ satisfy:
\begin{enumerate}
\item	 \label{reg1}
for $\lambda=\lambda_3$
the sets $\hor_i$ and $R_{\lambda_3}(\hor_j)$ intersect regularly for $i,j\in \{a,b\}$;
\item  \label{reg2}
for $\lambda=\lambda_2$
the intersection $\hor_a \cap R_{\lambda_2}(\hor_a )$ is irregular;
\item  \label{reg4}
for $\lambda=\lambda_1$
the sets $\hor_a$ and $R_{\lambda_1}(\hor_a )$ do not intersect at al;
\item	 \label{reg5}
for $\lambda=\lambda_1$ and $\lambda=\lambda_2$
the set $R_{\lambda}(\hor_b)$ intersects both 
 $\hor_b$ and  $\hor_a$ regularly.
\end{enumerate}
\end{proposition}

\begin{figure}
\begin{center}
\includegraphics[width=14cm]{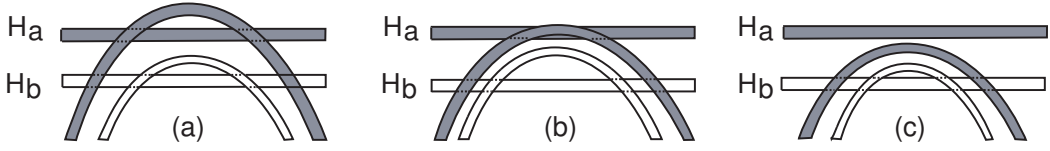}
\end{center}
\caption{\small Strips in Proposition~\ref{regular_intersections}:  (a) $\lambda=\lambda_3$; (b) $\lambda=\lambda_2$; (c) $\lambda=\lambda_1$.
The position of the horizontal strips $\hor_a$ and $\hor_b$ does not depend on $\lambda$.
The maximum height of  the  horseshoe strips $R_\lambda(\hor_a)$ and $R_\lambda(\hor_b)$ decreases 
when $\lambda$ decreases, and the suspended  horseshoe in $R_\lambda(\hor_a)\cap \hor_a$ is destroyed.
}
\label{new2}
\end{figure}

\begin{proof}
As we have remarked before, the horizontal strips $\hor_n$ do not depend on the parameter $\lambda$. 
In particular, they are well defined for $\lambda=0$. 
Their image by the return map $R_0$ is no longer a horseshoe strip,  it is a horizontal strip across $\quadr_\vv$.
Since  for $\lambda=0$ the map $\Psi_{\ww \rightarrow \vv}^0$ is the identity (see \ref{subsecInvariantManifolds}) and
the network is asymptotically stable, it follows that the maximum height of $R_0(\hor_n)$ is no more than that of $\eta(\hor_n)$.
 
From  Lemma~\ref{lemaEtaHorizontalStrips} and the expressions \eqref{maxHn} and \eqref{minHn} for the maximum,  $h_n$, and minimum $m_n$, height of $\hor_n$,  we get that the
maximum height $h_n^\delta$ of  $\eta(\hor_n)$ is less than $m_n$ if
$$
L_n=\delta(P^1_\vv-P^2_\ww) + 2\tau(\delta-1)-2n\pi (\delta-1) < P^2_\vv-P^1_\ww .
$$
Since $\lim_{n\to\infty} L_n=-\infty$, there is $n_2$ such that $R_0(\hor_n)\cap \hor_n=\emptyset$,
for every $n>n_2$.
Therefore, since $R_\lambda$ is continuous on $\lambda$, for each $n>n_2$ there is $\lambda_\star(n)>0$ such that the maximum height of  $R_{\lambda(n)}(\hor_n)$ is less than that of $\hor_n$.
 
Let  $a=\max\{n_1,N(\lambda_3)\}$ for  $N(\lambda_3)$ from Lemma~\ref{Novo2}.
Then assertion~\eqref{reg1} is true for this $\hor_a$ and for any $\hor_n$ with $n>a$, by  Lemma~\ref{Novo2}. 
We obtain assertion~\eqref{reg4}  by taking $\lambda_1=\lambda_\star(a)$.
Assertion~\eqref{reg2} follows from the continuous dependence of $R_\lambda$ on $\lambda$.
Assertion~\eqref{reg5} holds for any $\hor_b$ with $b>N(\lambda_1)=b_0$ by Lemma~\ref{Novo2}.
\end{proof}

\subsection{Proof of Theorem~\ref{newhouse}}
In particular, Proposition~\ref{regular_intersections} implies that  there exists $ d \in[\lambda_2, \lambda_3)$ such that at $\lambda=  d $, the lower horizontal boundary of $R_\lambda(\hor_a)$ is tangent to the upper horizontal boundary of $\hor_a$. 
Analogously, there exists  $ c \in(\lambda_1, \lambda_2]$ such that at $\lambda=  c $, the upper horizontal boundary of $R_\lambda(\hor_a)$ is tangent to the lower horizontal boundary of $\hor_a$. 
There are infinitely many values of $\lambda$ in $[ c , d ]$ for which the map $R_\lambda$ has a homoclinic tangency associated to a periodic point see  \cite{ PT, YA}.

The rigorous  formulation of Theorem~\ref{newhouse} consists of Proposition~\ref{regular_intersections} , with $\Delta_1= [ c , d ] $ as in the remarks above.
Since Proposition~\ref{regular_intersections} holds for any $\lambda_3>0$, it may be applied again with $\lambda_3$ replaced by $\lambda_1$, and 
the argument may be repeated recursively to obtain a sequence of disjoint intervals $\Delta_n$.
\Qed

\subsection{Proof of Corollary~\ref{Consequences}}
The  $\lambda$ dependence of the position of the horseshoe strips is not a translation in our case, but after Proposition~\ref{regular_intersections} the constructions  of Yorke and Alligood \cite{YA} and of Palis and Takens \cite{PT} can be carried over.
Hence, when the parameter
$\lambda$ varies between two consecutive regular intersections of strips, the bifurcations of the first return map can be described as the one-dimensional parabola-type  map in \cite{PT}.
Following   \cite{PT, YA}, the bifurcations for $\lambda\in[\lambda_1,\lambda_3]$ are:
  \begin{itemize}
 \item
for $\lambda> d $:  the restriction of the map $R_{\lambda_3}$ to the non-wandering set on $\hor_a$ is conjugate to the Bernoulli shift of two symbols  and it no longer bifurcates as $\lambda$ increases. 
 \item  \label{fixedPt}
 at $\lambda \in (c , d ) $: a  fixed point with multiplier equal to $\pm 1$ appears at a tangency of the horizontal boundaries. 
It undergoes a period-doubling bifurcation. A cascade of period-doubling bifurcations leads to chaotic dynamics which alternates with stability windows and the bifurcations stop at $\lambda= d $;
  \item 
at $\lambda< c $: trajectories with initial conditions in $\hor_a$ approach the network and might be attracted to another basic set.
 \end{itemize}
  \Qed
  
  \subsection{Proof of Corollary~\ref{Sinks}}
If the first return map $R_\lambda$ is area-contracting, then the fixed point that appears  for 
$\lambda \in (c , d )$ is attracting for  the parameter in an open interval.
 
This attracting  fixed point bifurcates to a sink of period 2 at a bifurcation parameter  $\lambda> c $ close to $ c $.
 This stable orbit  undergoes a second flip bifurcation, yielding an orbit of period 4. This process continues to an accumulation point in parameter space at which attracting orbits of period $2^k$ exist, for all $k\in \NN$.  
 This completes the proof of Corollary~\ref{Sinks} .
\Qed
  
 Finally, we describe a setting in which the fixed point that appears  for 
$\lambda \in (c , d )$ can be shown to be attracting.
Recall that the set  $W^u_{\loc}(\ww)\cap \In(\vv)$ is  the graph  of  $y=\wuw(x,\lambda)$.
Suppose the transition map $\Psi_{\ww \rightarrow \vv}^\lambda:\Out(\ww)\longrightarrow \In(\vv)$ is given by $\Psi_{\ww \rightarrow \vv}^\lambda(x,y)=\left(x,y+\wuw(x,\lambda)\right)$, 
which is consistent with Section~ \ref{subsecInvariantManifolds}.
Since 
$$ 
D \Psi_{\ww \rightarrow \vv}^\lambda(x,y)=\left(\begin{array}{cc} 1&0\\ 
\frac{\partial\wuw}{\partial x}(x)&1\end{array}\right) 
\quad\mbox{and}\quad
D\eta(x,y)=
\left(\begin{array}{cc} 1&\dpt\frac{-K}{y}\\ \\ 0&\delta y^{\delta-1}\end{array}\right)
$$  
then $$\det D R_\lambda(x,y)= \det D \Psi_{\ww \rightarrow \vv}^\lambda\left(\eta(x,y)\right) \cdot \det D\eta(x,y)=\delta y^{\delta-1}.$$
For sufficiently small $y$ (in $\hor_n$ with sufficiently large $n$) this is less than 1, and hence $R_\lambda$ is contracting.
However, if the first coordinate of $\Psi_{\ww \rightarrow \vv}^\lambda(x,y)$ depends on $y$, then 
$ \det D \Psi_{\ww \rightarrow \vv}^\lambda\left(\eta(x,y)\right)$ will contain terms that depend on 
$x+K\ln y$ and a more careful analysis will be required.

\section*{Acknowledgements}
We thank the two  anonymous referees for the careful reading of the manuscript and the useful corrections and suggestions
that helped to improve the work and for providing additional references.


\begin{thebibliography}{99}

\bibitem{ABS} V.S. Afraimovich, V.V. Bykov, L.P. Shilnikov, \emph{ The origin and structure of the Lorenz attractor}, Sov. Phys. Dokl., 22, 253--255, 1977

\bibitem{Afraimovich83} V.S. Afraimovich, L.P. Shilnikov, \emph{Strange attractors and quasiattractors},
in: G.I. Barenblatt, G. Iooss, D.D. Joseph (Eds.), Nonlinear Dynamics and
Turbulence, Pitman, Boston, 1--51, 1983

\bibitem{ACL NONLINEARITY} M.A.D. Aguiar, S.B.S.D. Castro, I.S. Labouriau, 
\emph{Dynamics near a heteroclinic network,} Nonlinearity, No.  18, 391--414, 2005

\bibitem{ACL2} M.A.D. Aguiar, S.B.S.D. Castro, I.S. Labouriau, 
\emph{Simple Vector Fields with Complex Behaviour}, Int. J. 
Bif. Chaos, Vol. {16}, No. 2, 369--381,  2006


\bibitem{ALR} M.A.D. Aguiar, I.S. Labouriau, A.A.P. Rodrigues, \emph{Switching near a heteroclinic network of rotating nodes}, Dynamical Systems, Vol. 25,  1, 75--95, 2010

\bibitem{AGH} D. Armbruster, J. Guckenheimer, P. Holmes, \emph{Heteroclinic cycles and modulated travelling waves in systems with O(2) symmetry}, Physica D, No. 29, 257--282, 1988

\bibitem{BessaRodrigues} M. Bessa, A.A.P. Rodrigues, \emph{Dynamics of conservative Bykov cycles: tangencies, generalized Cocoon bifurcations and elliptic solutions}, preprint, 2015

\bibitem{Bowen75} R. Bowen, \emph{Equilibrium States and the Ergodic Theory of Anosov Diffeomorphisms}, Lect. Notes in Math, Springer, 1975


\bibitem{Bykov93} V. V. Bykov,  \emph{The bifurcations of separatrix contours and chaos}, {Physica D}, 62, No.1-4, 290--299, 1993

\bibitem{Bykov99} V. V. Bykov, \emph{On systems with separatrix contour containing two saddle-foci}, {J. Math. Sci.}, {95}, , 2513--2522, 1999


\bibitem{Bykov} V. V. Bykov, \emph{Orbit Structure in a Neighbourhood of a Separatrix Cycle Containing Two Saddle-Foci}, {Amer. Math. Soc. Transl}, {200}, 87--97, 2000


\bibitem{Chillingworth} D.R.J. Chillingworth, \emph{Generic multiparameter bifurcation from a manifold},
Dyn. Stab. Syst., Vol. 15,  2, 101--137, 2000

\bibitem{Colli} E. Colli, \emph{Infinitely many coexisting strange attractors}, Ann. Inst. H. Poincar\'e, Anal. Non Lin\'eaire, 15, 539--579, 1998

\bibitem{Deng} B. Deng, \emph{Exponential expansion with Shilnikoc saddle-focus}, J. Diff. Eqns, 82, 156--173, 1989 

\bibitem{DIK} F. Dumortier, S. Ib\'a\~nez, H. Kokubu, \emph{Cocoon bifurcation in three-dimensional reversible vector fields}, Nonlinearity 19, 305--328, 2006

\bibitem{DIKS} F. Dumortier, S. Ib\'a\~nez, H. Kokubu, C. Sim\'o, \emph{About the unfolding of a Hopf-zero singularity}, {Disc. Cont. Dyn. Syst.}, { 33}, 10, 4435--4471, 2013

\bibitem{GavS} N.K. Gavrilov, L.P. Shilnikov, \emph{On three-dimensional dynamical systems close to systems with a structurally unstable homoclinic curve}, Part I.  Math. USSR Sbornik,  17, 467--485; 1972;  Part II. {\em ibid}, {19}, 139--156, 1973 


\bibitem{GSpa} G. Glendinning, C. Sparrow, \emph{T-points: a codimension two heteroclinic bifurcation}, {J. Statist. Phys.}, {43}, 479--488, 1986

\bibitem{GS} M. Golubitsky, I. Stewart, \emph{The Symmetry Perspective}, Birkhauser, 2000

\bibitem{GST1993} S. V. Gonchenko, L.P. Shilnikov, D.V. Turaev, \emph{On models with non-rough Poincar\'e homoclinic curves}, Physica D, 62, 1--14, 1993

\bibitem{Gonchenko96} S.V. Gonchenko, L.P. Shilnikov, D.V. Turaev, \emph{Dynamical phenomena in systems with structurally unstable Poincar\'e homoclinic orbits},  Chaos 6, No. 1, 15--31, 1996


\bibitem{Gonchenko2007} S.V. Gonchenko, L.P. Shilnikov, D.V. Turaev, \emph{Quasiattractors and Homoclinic Tangencies}, Computers  Math.  Applic.  Vol.  34, No.  2-4,  195--227,  1997 


\bibitem{GST2007} S. V. Gonchenko, L. P. Shilnikov, D.Turaev, \emph{Homoclinic tangencies of arbitrarily high orders in conservative and dissipative two-dimensional maps},  {Nonlinearity}, { 20}, 241--275, 2007


\bibitem{Gonchenko2012} S.V. Gonchenko, I.I. Ovsyannikov, D.V. Turaev, \emph{On the effect of invisibility of stable periodic orbits at homoclinic bifurcations}, Physica D, 241, 1115--1122, 2012

\bibitem{GTS2001}  S.V. Gonchenko, D.V. Turaev, L.P. Shilnikov, \emph{Homoclinic tangencies of an arbitrary order in Newhouse domains}, in Itogi Nauki Tekh., Ser. Sovrem. Mat. Prilozh. 67, 69--128, 1999 [English translation in J. Math. Sci. 105, 1738--1778 (2001)].


 
\bibitem{GH} J. Guckenheimer, P. Holmes, \emph{Nonlinear and Bifurcations of Vector Fields}, Applied Mathematical Sciences, No. 42, Springer-Verlag, 1983

\bibitem{Homburg} A. J. Homburg, \emph{Periodic attractors, strange attractors and hyperbolic dynamics near homoclinic orbit to
a saddle-focus equilibria}. Nonlinearity 15, 411--428, 2002

\bibitem{HS} A.J. Homburg, B. Sandstede, \emph{Homoclinic and Heteroclinic Bifurcations in Vector Fields}, Handbook of Dynamical Systems, Vol. 3, North Holland, Amsterdam, 379--524, 2010


\bibitem{Kiriki} S. Kiriki, T. Soma, \emph{Existence of generic cubic homoclinic tangencies for H\'enon maps}, Ergodic Theory and Dynamical Systems, 33, 1029--1051, 2013


\bibitem{KR} V. Kirk, A.M. Rucklidge, \emph{The effect of symmetry breaking on the dynamics near a
structurally stable heteroclinic cycle between equilibria and a periodic orbit}, Dyn. Syst. Int. J. 23, 43--74, 2008

\bibitem{KLW} J. Knobloch, J.S.W. Lamb, K.N. Webster, \emph{Using Lin's method to solve Bykov's problems}, J. Diff. Eqs., 257(8), 2984--3047, 2014

\bibitem{KLW2015} J. Knobloch,  J.S.W. Lamb, K.N. Webster, \emph{Shift dynamics near non-elementary T-points with real eigenvalues}, preprint, 2015

\bibitem{KM1} M. Krupa, I. Melbourne, \emph{Asymptotic Stability of Heteroclinic Cycles in Systems with Symmetry II, }Ergodic Theory
and Dynam. Sys., Vol. {15}, 121--147, 1995


\bibitem{LR} I.S. Labouriau, A.A.P. Rodrigues, \emph{Global generic dynamics close to symmetry}, J. Diff. Eqs., Vol. 253 (8), 2527--2557, 2012

\bibitem{LR_proc} I.S. Labouriau, A.A.P. Rodrigues, \emph{Partial symmetry breaking and heteroclinic tangencies}, in S. Ib\'a\~nez, J.S. P\'erez del R\'io, A. Pumari\~no and J.A. Rodr\'iguez (eds), Progress and challenges in dynamical systems, 281--299, 2013

\bibitem{LR3} I.S. Labouriau, A.A.P. Rodrigues, \emph{Dense heteroclinic tangencies near a Bykov cycle}, J. Diff. Eqs., 259, 5875--5902, 2015 

\bibitem{Lamb2005} J.S.W. Lamb, M.A. Teixeira, K.N. Webster, \emph{Heteroclinic bifurcations near Hopf-zero
bifurcation in reversible vector fields in $\textbf{R}^3$}, J. Diff. Eqs., 219, 78--115, 2005


\bibitem{MPR} I. Melbourne, M.R.E. Proctor and A.M. Rucklidge, \emph{A heteroclinic model of geodynamo reversals
and excursions}, Dynamo and Dynamics, a Mathematical Challenge (eds. P. Chossat, D. Armbruster
and I. Oprea, Kluwer: Dordrecht, 363--370, 2001


\bibitem{MV} L. Mora, M. Viana, \emph{Abundance of strange attractors}, Acta Math. 171, 1--71, 1993


\bibitem{Newhouse74} S.E. Newhouse, \emph{Diffeomorphisms with infinitely many sinks}, Topology 13
9--18, 1974

\bibitem{Newhouse79} S.E. Newhouse, \emph{The abundance of wild hyperbolic sets and non-smooth stable
sets for diffeomorphisms}, Publ. Math. Inst. Hautes Etudes Sci. 50, 101--151, 1979

\bibitem{OS} I. M. Ovsyannikov, L.P. Shilnikov, \emph{On systems with a saddle-focus homoclinic curve}, Math. USSR Sb., 58, 557--574, 1987 

\bibitem{PT} J. Palis, F. Takens, \emph{Hyperbolicity and sensitive chaotic dynamics at homoclinic bifurcations}, Cambridge University Press, Cambridge
Studies in Advanced Mathematics 35, 1993

\bibitem{PY} N. Petrovskaya, V. Yudovich, \emph{Homoclinic loops of Zal'tsman-Lorenz system}, Methods of Qualitative Theory of Diff. Equations, Gorkii, 73--83, 1980

\bibitem{RAOY} C. Robert, K. Alligood, E. Ott, J. Yorke, \emph{Explosions of chaotic sets}, Physica D, 44--61, 2000

\bibitem{Rodrigues2} A.A.P. Rodrigues, \emph{Persistent Switching near a Heteroclinic Model for the Geodynamo Problem}, Chaos, Solitons \& Fractals, 47, 73--86, 2013

\bibitem{Rodrigues3} A.A.P. Rodrigues, \emph{Repelling dynamics near a Bykov cycle}, J. Dyn. Diff. Eqs., Vol.25, Issue 3, 605--625, 2013 

\bibitem{Rodrigues4} A.A.P. Rodrigues, \emph{Moduli for heteroclinic connections involving saddle-foci and periodic solutions}, Discrete Contin. Dyn. Syst. A, Vol. 35(7), 3155--3182, 2015

\bibitem{LR2}  A.A.P. Rodrigues, I.S. Labouriau, \emph{Spiralling dynamics near heteroclinic networks}, Physica D, 268, 34-49, 2014

\bibitem{Ruck} A. M. Rucklidge, \emph{Chaos in a low-order model of magnetoconvection}, Physica D 62, 323--
337, 1993



\bibitem{Shilnikov65} L.P. Shilnikov, \emph{A case of the existence of a denumerable set of periodic motions}, Sov. Math. Dokl, No. 6, 163--166, 1965 

\bibitem{Shilnikov67} L.P. Shilnikov, \emph{On a Poincar\'e-Birkhoff problem}, Math. USSR Sb. 3, 353--371, 1967

\bibitem{Shilnikov67A} L.P. Shilnikov, \emph{The existence of a denumerable set of periodic motions in four dimensional space in an extended neighbourhood of a saddle-focus}, Sov. Math. Dokl., 8(1),  54--58, 1967

\bibitem{YA} J. A. Yorke, K. T. Alligood, \emph{Cascades of period-doubling bifurcations: A prerequisite for horseshoes}, Bull. Am. Math. Soc. (N.S.) 9(3), 319--322, 1983



\end{thebibliography}
\end{document}